\UseRawInputEncoding
\documentclass[12pt]{article}

\usepackage{dsfont}
\usepackage{amsfonts,amsmath,amsthm}
\usepackage{algorithm, algorithmic}
\usepackage{color}
\usepackage{graphicx}
\usepackage{stmaryrd}        

\usepackage[colorlinks=true]{hyperref}
\hypersetup{urlcolor=blue, citecolor=blue,linkcolor=blue}

\usepackage[paper=a4paper,dvips,top=2cm,left=2cm,right=2cm,
    foot=1cm,bottom=4cm]{geometry}

\begin{document}

\title{Quasi-determinant and right eigenvalues of dual quaternion matrices}
\author{ Chen Ling,\footnote{%
    Department of Mathematics, Hangzhou Dianzi University, Hangzhou, 310018, China
    ({\tt
macling@hdu.edu.cn}). This author's work was supported by Natural Science Foundation of China (No. 11971138).}
\and
Liqun Qi\thanks{Department of Applied Mathematics, The Hong Kong Polytechnic University, Hung Hom,
    Kowloon, Hong Kong; Department of Mathematics, Hangzhou Dianzi University, Hangzhou, 310018, China ({\tt maqilq@polyu.edu.hk}).}
}
\date{\today}
\maketitle

\begin{abstract}
Dual quaternion/complex matrices have important applications in brain science and multi-agent formation control. In this paper, we first study some basic properties of determinants of dual complex matrices, including Sturm theorem and Bloomfield-Watson inequality for dual complex matrices. Then, we show that every eigenvalue of a dual complex matrix must be the root of the characteristic polynomial of this matrix. With the help of the determinants of dual complex matrices, we introduce the concept of quasi-determinants of dual quaternion matrices, and show that every right eigenvalue of a dual quaternion matrix must be the root of the quasi-characteristic polynomial of this matrix, as well as the quasi-determinant of a dual quaternion Hermitian matrix is equivalent to the product of the square of the magnitudes of all eigenvalues. Our results are helpful for the further study of dual quaternion matrix theory, and their applications.

\medskip



\end{abstract}

\renewcommand{\Re}{\mathds{R}}
\newcommand{\rank}{\mathrm{rank}}
\renewcommand{\span}{\mathrm{span}}
\newcommand{\X}{\mathcal{X}}
\newcommand{\A}{\mathcal{A}}
\newcommand{\I}{\mathcal{I}}
\newcommand{\B}{\mathcal{B}}
\newcommand{\C}{\mathcal{C}}
\newcommand{\OO}{\mathcal{O}}
\newcommand{\e}{\mathbf{e}}
\newcommand{\0}{\mathbf{0}}
\newcommand{\dd}{\mathbf{d}}
\newcommand{\ee}{\mathbf{e}}
\newcommand{\ii}{\mathbf{i}}
\newcommand{\jj}{\mathbf{j}}
\newcommand{\kk}{\mathbf{k}}
\newcommand{\va}{\mathbf{a}}
\newcommand{\vb}{\mathbf{b}}
\newcommand{\vc}{\mathbf{c}}
\newcommand{\vg}{\mathbf{g}}
\newcommand{\vr}{\mathbf{r}}
\newcommand{\vt}{\rm{vec}}
\newcommand{\vx}{\mathbf{x}}
\newcommand{\vy}{\mathbf{y}}
\newcommand{\vu}{\mathbf{u}}
\newcommand{\vv}{\mathbf{v}}
\newcommand{\y}{\mathbf{y}}
\newcommand{\vz}{\mathbf{z}}
\newcommand{\T}{\top}

\newtheorem{Thm}{Theorem}[section]
\newtheorem{Def}[Thm]{Definition}
\newtheorem{Ass}[Thm]{Assumption}
\newtheorem{Lem}[Thm]{Lemma}
\newtheorem{Prop}[Thm]{Proposition}
\newtheorem{Cor}[Thm]{Corollary}
\newtheorem{example}[Thm]{Example}
\newtheorem{ReK}[Thm]{Remark}
\newtheorem{PropT}[Thm]{Property}

\section{Introduction}\label{Introd}
Dual numbers, dual complex numbers, dual quaternions, their vectors and matrices, as well as their applications have a long history. Dual numbers were introduced by Clifford \cite{Cl73} in 1873, then E. Study \cite{St91} applied dual numbers to measure the relative position of skew lines in three-dimensional space and defined a dual angle, whose standard and infinitesimal parts represent the angle and distance between lines, respectively. These started the study and applications of dual number algebra in kinematics, dynamics and robotics \cite{An98,An12,Fi98,GL87,PS07,PV09,Ud21,Wa21}.  Later, dual complex numbers, as a generalization of dual numbers, have found applications in spatial mechanisms \cite{CT96,CT97}. Recently, dual complex matrices and their singular value decomposition theory have found applications in brain science \cite{QACLL22,WDW24}. It is worth pointing out that, because dual quaternions can be used to represent coordinately a combination of rigid body's rotation and displacement, during the past decades, they have fruitful applications in engineering fields, such as 3D computer graphics, robotics control and computer vision \cite{BLH19,CKJC16,Da99,MKO14,MXXLY18,PV10,TRA11,WYL12,WZ14}. More recently, the eigenvalue theory of dual quaternion matrices have also found applications in  hand-eye calibration and multi-agent formation control, see \cite{CLQY24,QC24,QWL22}.

With the increasingly widespread application of dual quaternion matrices in engineering fields, research on related algebraic fundamental theories and computational methods has received increasing attention. Qi et al \cite{QLY21} introduced a total order for dual numbers, and extended $1$-norm, $\infty$-norm, and $2$-norm to dual quaternion vectors. In \cite{QL21}, it was shown an $m \times m$ dual quaternion Hermitian  matrix has exactly $m$ eigenvalues, which are all dual numbers. Singular value decomposition of qual quaternion matrices was also established in \cite{QL21}. Ling et al \cite{LHQ22} further studied the minimax principle and Weyl's type monotonicity inequality for singular values of dual quaternion matrices, as well as the polar decomposition and best low-rank approximations for dual quaternion matrices. In \cite{LQY22}, Ling et al  presented the minimax principle for eigenvalues of dual quaternion Hermitian matrices and Fan-Hoffman inequality for singular values of dual quaternion matrices, and studied the generalized inverses of dual quaternion matrices. In \cite{LHQF23}, von Neumann type trace inequality and Hoffman-Wielandt type inequality for general dual quaternion matrices were established. Qi et al \cite{QC23} studied the Jordan form of dual complex matrices with diagonalizable standard part, and the Jordan form of dual complex matrices with a Jordan block standard part. Based on these, the authors further proposed a description of the eigenvalues of a general square dual complex matrix in \cite{QC23}. In \cite{QC24}, Qi et al introduced the determinant, the characteristic polynomial and the supplement matrices for dual Hermitian matrices, and presented a practical method for computing eigenvalues of dual Hermitian matrices.

Determinant is an important concept and tool in linear algebra, which can be used to study the linear equation systems, the eigenvalues and the subspace theory of matrices.  Due to the noncommutative multiplication of dual quaternions, the study on the determinant theory of dual quaternion matrices is very difficult \cite{Kyr12, LSZ20, SWC11}. In this paper, we first study the determinant properties of dual complex matrices, and then introduce the concept of quasi-determinant of dual quaternion matrices. Based upon these, the relationship between the right eigenvalues of dual quaternion matrices and their quasi-characteristic polynomials are studied.

Our paper consists of six sections. In Section \ref{Prelim}, we first present some preliminary knowledge on quaternions, dual numbers, dual quaternions, and dual quaternion matrices.
Then, in Section \ref{Deter-DNMat}, we introduce the concept of determinant of dual complex matrices, and study some basic properties of the determinant of dual complex matrices.
In Section \ref{Eigen-DNMat}, we first study the relationship between the eigenvalues of dual complex matrices and their characteristic polynomials. Then we show that, the eigenvalue of a dual complex matrix must be the root of the characteristic polynomial of this matrix, however, a characteristic root of a dual number matrix may not necessarily be its eigenvalue.
With help of the singular value decomposition of dual complex matrices and the minimax principle of dual complex Hermitian matrices, we also obtain several important properties of the determinant of dual complex matrices, including the important Bloomfield-Watson inequality. After introducing the concept of the quasi-determinant of dual quaternion matrices, in Section \ref{DET-DQ}, we show that the right eigenvalue of a dual quaternion matrix must be the root of the quasi-characteristic polynomial of this matrix. Furthermore, we also show that the determinant of a dual quaternion Hermitian matrix is equivalent to the product of the square of the magnitudes of all eigenvalues. Finally, we complete this paper by making some final remarks in Section \ref{Conclusion}.

\section{Preliminaries}\label{Prelim}
\subsection{Quaternions}\label{Quat}
Let $\mathbb{R}$ denote the field of the real numbers, and denote by $\mathbb Q$ the four-dimensional vector space of the quaternions over $\mathbb{R}$. A quaternion $q\in \mathbb{Q}$ has the form
$q = q_0 + q_1\ii + q_2\jj + q_3\kk$, where $q_0, q_1, q_2, q_3\in \mathbb{R}$, $\ii, \jj$ and $\kk$ are three imaginary units of quaternions, satisfying $\ii^2 = \jj^2 = \kk^2 =\ii\jj\kk = -1$, $\ii\jj = -\jj\ii = \kk$, $\jj\kk = - \kk\jj = \ii$, $\kk\ii = -\ii\kk = \jj$. The real part of $q$ is ${\rm Re}(q):= q_0$, 
the imaginary part of $q$ is ${\rm Im}(q):= q_1\ii + q_2\jj +q_3\kk$, the conjugate of $q$ is $\bar q := q_0 -q_1\ii -q_2\jj -q_3\kk$, and the norm of $q$ is $|q|:=\sqrt{\bar qq}=\sqrt{q_0^2+q_1^2+q_2^2+q_3^2}$.
A quaternion is called imaginary if its real part is zero. It is obvious that, if $q_2=q_3=0$, then $q$ is a complex number; if ${\rm Im}(q)=0$, i.e., $q_1=q_2=q_3=0$, then $q$ is a real number. The field of complex numbers is denoted by $\mathbb{C}$. It is obvious that $\mathbb{R}\subset\mathbb{C}\subset\mathbb{Q}$. The multiplication of quaternions satisfies the distribution law, but is noncommutative. 

\subsection{Dual  numbers}
The set of dual numbers is denoted by $\widehat{\mathbb{R}}$. A dual number $q\in \widehat{\mathbb{R}}$ has the form $a = a_{\sf st} + a_{\sf in}\epsilon$, where $a_{\sf st},a_{\sf in}\in \mathbb{R}$,  and $\epsilon$ is the infinitesimal unit, satisfying $\epsilon^2 = 0$ but $\epsilon\neq 0$.  We call $a_{\sf st}$ the standard part of $a$, and $a_{\sf in}$ the infinitesimal part of $a$.   If $a_{\sf st}\neq 0$, then we say that $a$ is appreciable, otherwise, we say that $a$ is infinitesimal. The infinitesimal unit $\epsilon$ is commutative in multiplication with real numbers.

In \cite{QLY21}, a total order was introduced for dual numbers. Suppose $a = a_{\sf st} + a_{\sf in}\epsilon, b = b_{\sf st} + b_{\sf in}\epsilon \in \widehat{\mathbb{R}}$.  We have $b < a$ if $b_{\sf st} < a_{\sf st}$, or $b_{\sf st} = a_{\sf st}$ and $b_{\sf in} < a_{\sf in}$.  It is obvious that $b = a$, if and only if $b_{\sf st} = a_{\sf st}$ and $b_{\sf in} = a_{\sf in}$.   Thus, if $a > 0$, we say that $a$ is a positive dual number; and if $a \ge 0$, we say that $a$ is a nonnegative dual number.  Denote the set of nonnegative dual numbers by $\widehat{\mathbb{R}}_+$, and the set of positive dual numbers by $\widehat{\mathbb{R}}_{++}$.

The absolute value \cite{QLY21} of $a \in \widehat{\mathbb{R}}$ is defined by
\begin{equation} \label{e5}
|a| = \left\{ \begin{array}{ll}|a_{\sf st}| + {\rm sgn}(a_{\sf st})a_{\sf in}\epsilon, & {\rm if~}  a_{\sf st} \not = 0, \\
|a_{\sf in}|\epsilon, & {\rm otherwise},  \end{array}  \right.
\end{equation}
where for any $u \in \mathbb R$,
$${\rm sgn}(u) =  \left\{ \begin{aligned} 1, & \ {\rm if}\  u > 0, \\ 0, &   \ {\rm if}\  u = 0, \\
-1, &   \ {\rm if}\  u < 0.  \end{aligned}  \right. $$
For given $a = a_{\sf st} + a_{\sf in}\epsilon, b = b_{\sf st} + b_{\sf in}\epsilon \in \widehat{\mathbb{R}}$, we define
\begin{equation} \label{e1}
 a+ b =a_{\sf st}+b_{\sf st} +(a_{\sf in}+b_{\sf in})\epsilon,~~~~ab = a_{\sf st}b_{\sf st} +(a_{\sf st}b_{\sf in}+a_{\sf in} b_{\sf st})\epsilon.\end{equation}
If $a$ is appreciable, then $a$ is nonsingular and $a^{-1} = a_{\sf st}^{-1} - a_{\sf st}^{-2}a_{\sf in} \epsilon$, which satisfies $aa^{-1}=1$. If $a$ is infinitesimal, then $a$ is not nonsingular. Furthermore, for any positive integer $k$, we have
 \begin{equation} \label{e2}
a^k = a_{\sf st}^k + ka_{\sf st}^{k-1}a_{\sf in} \epsilon.
\end{equation}
If $a$ is appreciable, then $k$ in (\ref{e2}) can be taken as a negative integer.
\begin{Prop}\label{P6.5}\cite{QLY21}
Suppose $a,b\in \widehat{\mathbb{R}}_+$. Then, we have $ab \in \widehat{\mathbb{R}}_{+}$. 







\end{Prop}

\subsection{Dual quaternions and dual quaternion matrices}
Denote by $\widehat{\mathbb Q}$ the set of dual quaternions.
A dual quaternion $q$ has the form $q = q_{\rm st} + q_{\rm in}\epsilon$,
where $q_{\rm st}, q_{\rm in} \in \mathbb {Q}$ are the standard part and the infinitesimal part of $q$, respectively.  If $q_{\rm st} \not = 0$, then we say that $q$ is appreciable, otherwise, we say that $q$ is infinitesimal. The conjugate of $q= q_{\rm st} + q_{\rm in}\epsilon$ is $\bar q = \bar q_{\rm st} + \bar q_{\rm in}\epsilon$. See \cite{BK20, CKJC16, Ke12}. We can derive that $q$ is invertible if and only if  $q$ is appreciable. In this case, we have $q^{-1} = q_{\rm st}^{-1} - q_{\rm st}^{-1}q_{\rm in} q_{\rm st}^{-1} \epsilon$. If $q_{\rm st}, q_{\rm in} \in \mathbb {C}$, we say that $q$ is a dual complex number. The set of dual complex numbers is denoted by $\widehat{\mathbb{C}}$. It is obvious that $\widehat{\mathbb R}\subset\widehat{\mathbb C}\subset \widehat{\mathbb Q}$. The magnitude of $q\in \widehat{\mathbb Q}$ is defined as
\begin{equation} \label{e7}
\displaystyle|q| := \left\{ \begin{array}{ll} |q_{\rm st}| +\displaystyle {(q_{\rm st}\bar q_{\rm in}+q_{\rm in} \bar q_{\rm st}) \over 2|q_{\rm st}|}\epsilon, & \ {\rm if}\  q_{\rm st} \not = 0, \\
|q_{\rm in}|\epsilon, &  \ {\rm otherwise},
\end{array} \right.
\end{equation}
which is a dual number. 

Denote by $\widehat{\mathbb Q}^{m\times n}$ the set of $m\times n$ dual quaternion matrices. Then $A\in \widehat{\mathbb Q}^{m\times n}$ can
be written as $A = A_{\rm st} + A_{\rm in}\epsilon$, where $A_{\rm st}, A_{\rm in}\in \mathbb{Q}^{m\times n}$ are the standard part and the infinitesimal part of $A$, respectively.  If $A_{\rm st} \not = O$, we say that $A$ is appreciable, otherwise, we say that $A$ is infinitesimal. 
It is obvious that when $n=1$, dual quaternion matrix $A$ reduces to the dual quaternion column vector with $m$ components. In this case, $\widehat{\mathbb Q}^{m\times 1}$ is abbreviated as $\widehat{\mathbb Q}^m$. For given ${\bf u}=(u_1,u_2,\ldots,u_m)^\top$ and ${\bf v}=(v_1,v_2,\ldots,v_m)^\top$ in $\widehat{\mathbb Q}^m$, denote by $\langle {\bf u},{\bf v}\rangle$ the dual quaternion-valued inner product, i.e., $\langle {\bf u},{\bf v}\rangle=\sum_{i=1}^m\bar v_iu_i$. Accordingly, the orthogonality and orthonormality of vectors in $\widehat{\mathbb Q}^m$ can be defined similarly to vectors in $\mathbb{Q}^m$.
Denote by $[{\bf 0}]$ the set of all infinitesimal vectors in $\widehat{\mathbb Q}^m$.  
For given $A = A_{\sf st} + A_{\sf in} \epsilon = (a_{ij}) \in {\widehat{\mathbb {Q}}}^{m \times n}$, the Frobenius norm of $A$, which is a dual number, is defined by
\begin{equation}\label{FNorm-DQM}
\|A \|_F = \left\{\begin{array}{ll}\displaystyle\sqrt{\sum_{i=1}^m \sum_{j=1}^n |a_{ij}|^2}, & \ {\rm if}\ A_{\sf st} \not = O, \\
\|A_{\sf in}\|_F\epsilon, & \ {\rm otherwise.} \end{array}\right.
\end{equation}

For given $A\in \widehat{\mathbb Q}^{m\times n}$, the transpose of $A$ is denoted as $A^\top = (a_{ji})$, the conjugate of $A$ is denoted as ${\bar A} = (\bar a_{ij})$, and the conjugate transpose of $A$ is denoted as $A^*=(\bar a_{ji})=\bar A^\top$. It is obvious that $A^\top = A_{\rm st}^\top + A_{\rm in}^\top\epsilon$, $\bar A = \bar A_{\rm st}+\bar A_{\rm in}\epsilon$ and $A^* = A_{\rm st}^* + A_{\rm in}^*\epsilon$. A square matrix $A\in \widehat{\mathbb Q}^{m\times m}$ is called nonsingular (invertible) if $AB = BA = I_m$ for some $B\in \widehat{\mathbb Q}^{m\times m}$. In this case, we denote $A^{-1} = B$, and have
\begin{equation}\label{Inver-mat}
A^{-1}=A_{\sf st}^{-1}-A_{\sf st}^{-1}A_{\sf in}A_{\sf st}^{-1}\epsilon.
\end{equation}
A square matrix $A\in \widehat{\mathbb Q}^{m\times m}$ is called normal if $AA^*=A^*A$, Hermitian if $A^*=A$, and unitary if $A$ satisfies $A^*A=I_m$. Similarly, we say that $A\in \widehat{\mathbb C}^{m\times k} ~(k\leq m)$ is partially unitary, if $A$ satisfies $A^*A=I_k$. We have $(AB)^{-1}=B^{-1}A^{-1}$ if $A$ and $B$ are nonsingular, and $(A^*)^{-1}=(A^{-1})^*$ if $A$ is nonsingular. It is easy to see that a square matrix $U=[{\bf u}_1,{\bf u}_2,\ldots,{\bf u}_m]\in \widehat{\mathbb Q}^{m\times m}$ is unitary if and only if $\{{\bf u}_1,{\bf u}_2,\ldots,{\bf u}_m\}$ form an orthonormal basis of $\widehat{\mathbb Q}^m$, i.e., it is orthonormal and any vector ${\bf x}$ in $\widehat{\mathbb Q}^m$ can be written as ${\bf x}=\sum_{i=1}^m{\bf u}_i\alpha_i$ for some $\alpha_1,\alpha_2,\ldots,\alpha_m\in \widehat{\mathbb Q}$. A dual quaternion  Hermitian matrix $A \in \widehat{\mathbb {Q}}^{m \times m}$ is called positive semidefinite if $\vx^\top A\vx \ge 0$ for any $\vx \in \widehat{\mathbb {Q}}^m$; $A$ is called positive definite if for any $\vx \in \widehat{\mathbb {C}}^m$ with $\vx$ being appreciable,  we have $\vx^*A\vx > 0$ and is appreciable. For given $A,B \in \widehat{\mathbb {Q}}^{m \times m}$, if there is an $m\times m$ invertible matrix $P\in \widehat{\mathbb {Q}}^{m \times m}$ such that $A=P^{-1}BP$, then we say that $A$ and $B$ are similar, and denote $A\sim B$.
\begin{Prop}
\label{Prop3.4}
(Corollary 3.10 in \cite{LHQ22}) Suppose that $U \in \widehat{\mathbb{Q}}^{m \times k}$ is partially unitary, $k < m$.    Then there is a $V \in \widehat{\mathbb{Q}}^{m \times (m-k)}$ such that $(U, V) \in \widehat{\mathbb{Q}}^{m \times m}$ is unitary.
\end{Prop}

\begin{Thm}\label{SVD-DQM}\cite{QL21}
For given $A\in \widehat{\mathbb{Q}}^{m\times n}$, there exists a unitary matrix $U\in \widehat{\mathbb{Q}}^{m\times m}$ and a unitary matrix $V\in\widehat{\mathbb{Q}}^{n\times n}$, such that
\begin{equation}\label{SVDDQMEQ}
A=U\left(\begin{array}{cc}\Sigma_s&O\\
O&O
\end{array} \right)_{m\times n}V^*,
\end{equation}
where $\Sigma_s\in \widehat{\mathbb{Q}}^{s\times s}$ is a diagonal matrix, taking the form
$\Sigma_s={\rm diag} (\sigma_1,\ldots, \sigma_r,\ldots,\sigma_s)$, $r \leq s\leq t:={\min}\{m, n\}$, $\sigma_1\geq \sigma_2\geq\ldots\geq\sigma_r$ are positive appreciable dual numbers, and $\sigma_{r+1}\geq \sigma_{r+2}\geq\ldots\geq\sigma_s$ are positive infinitesimal dual numbers. Counting possible multiplicities of the diagonal entries, the form $\Sigma_s$ is unique.
\end{Thm}

In Theorem \ref{SVD-DQM}, the dual numbers $\sigma_1, \cdots, \sigma_s$ and possibly $\sigma_{s+1} = \cdots = \sigma_t=0$, if $s < t$, are called the singular values of $A$, the integer $s$ is called the rank of $A$, and the integer $r$ is called the appreciable rank of $A$.   We denote the rank of $A$ by rank$(A)$, and the appreciable rank of $A$ by Arank$(A)$.






\section{Determinant of dual complex matrices}\label{Deter-DNMat}
Denote by $\widehat{\mathbb C}^{m\times n}$ the set of $m\times n$ dual complex matrices. Due to the commutative multiplication of dual complex numbers, we present the following definition of the determinant of dual complex matrices.
\begin{Def}\label{Def-Det}
Let $A=(a_{ij})\in \widehat{\mathbb{C}}^{m\times m}$. 
The determinant of $A$ is defined by
$$
{\rm det}[A]=\mathop{\sum}\limits_{(j_1,j_2,\ldots,j_m)}(-1)^{\tau(j_1,j_2,\ldots,j_m)}a_{1j_1}a_{2j_2}\cdots a_{mj_m},
$$
where the summation, of $m!$ terms, being extended over all permutations $(j_1,j_2,\ldots,j_m)$ of column indexes of the elements $a_{ij}$, and $\tau(j_1,j_2,\ldots,j_m)$ represents the number of the reverse order pairs of the permutation $(j_1,j_2,\ldots,j_m)$.
\end{Def}

From Definition \ref{Def-Det}, it is obvious that ${\rm det}[\bar A]=\overline{{\rm det}[A]}$, and ${\rm det}[A]=\prod_{i=1}^ma_{ii}$ for any given diagonal matrix $A={\rm diag}(a_{11},a_{22},\ldots,a_{mm})\in \widehat{\mathbb{C}}^{m\times m}$. 
Moreover, due to the commutativity and distributive law of the multiplication of dual complex numbers, by Definition \ref{Def-Det}, we have
\begin{PropT}\label{Five-Prop-Det}
Let $A=({\bf a}_1;\ldots;{\bf a}_i;\ldots;{\bf a}_j;\ldots;{\bf a}_m)\in \widehat{\mathbb{C}}^{m\times m}$ with $i<j$ and ${\bf a}_k$ being the $k$-th row of $A$ for $k=1,2,\ldots,m$, we have

(a) ${\rm det}[A^\top]={\rm det}[A]$;

(b) ${\rm det}[B]=-{\rm det}[A]$, where $B=({\bf a}_1;\ldots;{\bf a}_{i-1};{\bf a}_j;{\bf a}_{i+1};\ldots;{\bf a}_{j-1};{\bf a}_i;{\bf a}_{j+1};\ldots;{\bf a}_m)$;

(c) ${\rm det}[B]=\alpha{\rm det}[A]$, where $B=({\bf a}_1;\ldots;{\bf a}_{i-1};\alpha{\bf a}_i;{\bf a}_{i+1};\ldots;{\bf a}_m)$ and $\alpha\in \widehat{\mathbb{C}}$;

(d) ${\rm det}[B]={\rm det}[A]$, where $B=({\bf a}_1;\ldots;{\bf a}_i;\ldots;{\bf a}_{j-1};{\bf a}_j+\alpha{\bf a}_i;{\bf a}_{j+1};\ldots;{\bf a}_m)$ and $\alpha\in \widehat{\mathbb{C}}$;

(e) ${\rm det}[C]={\rm det}[A]+{\rm det}[B]$, where $$B=({\bf a}_1;\ldots;{\bf a}_{i-1};{\bf b}_i;{\bf a}_{i+1};\ldots;{\bf a}_m)~~{\rm and}~~C=({\bf a}_1;\ldots;{\bf a}_{i-1};{\bf a}_i+{\bf b}_i;{\bf a}_{i+1};\ldots;{\bf a}_m).$$
\end{PropT}

Similar to the determinant of the sum of two complex square matrices, we also have the following expansion formula.
\begin{Prop}\label{Prop-+} Let $A, B\in \widehat{\mathbb{C}}^{m\times m}$. It holds that
\begin{equation}\label{A+B}
\begin{array}{l}
{\rm det}[A+B]\\
={\rm det}[A]+\mathop{\sum}\limits_{1\leq i_1\leq m}\mathop{\sum}\limits_{1\leq j_1\leq m}{\rm det}\big[\hat{A}_C\left(\begin{array}{c}i_1\\j_1\end{array}\right)\big]{\rm det}\big[B\left(\begin{array}{c}i_1\\j_1\end{array}\right)\big]\\
~~+\mathop{\sum}\limits_{1\leq i_1<i_2\leq m}\mathop{\sum}\limits_{1\leq j_1<j_2\leq m}{\rm det}\big[\hat{A}_C\left(\begin{array}{c}i_1i_2\\j_1j_2\end{array}\right)\big]{\rm det}\big[B\left(\begin{array}{c}i_1i_2\\j_1j_2\end{array}\right)\big]+\ldots\\
~~+\mathop{\sum}\limits_{1\leq i_1<\ldots <i_{m-1}\leq m}\mathop{\sum}\limits_{1\leq j_1<\ldots<j_{m-1}\leq m}{\rm det}\big[\hat{A}_C\left(\begin{array}{c}i_1\cdots i_{m-1}\\j_1\cdots j_{m-1}\end{array}\right)\big]{\rm det}\big[B\left(\begin{array}{c}i_1\cdots i_{m-1}\\j_1\cdots j_{m-1}\end{array}\right)\big]+{\rm det}[B],
\end{array}
\end{equation}
where ${\rm det}\big[D\left(\begin{array}{c}i_1\cdots i_{k}\\j_1\cdots j_{k}\end{array}\right)\big]$ represents the $k$-order determinant composed of $k^2$ elements at the intersection of rows $i_1,i_2,\ldots,i_k$ and columns $j_1,j_2,\ldots, j_k$ in the original order in $D$, and ${\rm det}\big[\hat{D}_C\left(\begin{array}{c}i_1\cdots i_{k}\\j_1\cdots j_{k}\end{array}\right)\big]$ represents the algebraic cofactor of ${\rm det}\big[D\left(\begin{array}{c}i_1\cdots i_{k}\\j_1\cdots j_{k}\end{array}\right)\big]$ in $D$.
\end{Prop}

\begin{proof}
It follows from (e) in Property \ref{Five-Prop-Det}.
\end{proof}
\begin{Prop}\label{A-det-Prop-1} Let $A=A_{\sf st}+A_{\sf in}\epsilon$, where $A_{\sf st}, A_{\sf in}\in \mathbb{C}^{m\times m}$. Then we have
\begin{equation}\label{Det-A}
{\rm det}[A]={\rm det}[A_{{\sf st}}]+\left(\sum_{i=1}^m{\rm det}[A(i)]\right)\epsilon,
\end{equation}
where
$$
A(i)=\left(\begin{array}{cccc}
(a_{11})_{\sf st}&(a_{12})_{\sf st}&\cdots&(a_{1m})_{\sf st}\\
\vdots&\vdots&\ddots&\vdots\\
(a_{i-11})_{\sf st}&(a_{i-12})_{\sf st}&\cdots&(a_{i-1m})_{\sf st}\\
(a_{i1})_{\sf in}&(a_{i2})_{\sf in}&\cdots&(a_{im})_{\sf in}\\
(a_{i+11})_{\sf st}&(a_{i+12})_{\sf st}&\cdots&(a_{i+1m})_{\sf st}\\
\vdots&\vdots&\ddots&\vdots\\
(a_{m1})_{\sf st}&(a_{m2})_{\sf st}&\cdots&(a_{mm})_{\sf st}\\
\end{array}
\right),~~~~~i=1,2,\ldots,m.
$$
\end{Prop}

\begin{proof}
Since $\Pi_{i=1}^ma_i=\Pi_{i=1}^m(a_i)_{\sf st}+\big((a_1)_{\sf in}(a_2)_{\sf st}(a_3)_{\sf st}\cdots (a_m)_{\sf st}+(a_1)_{\sf st}(a_2)_{\sf in}(a_3)_{\sf st}\cdots (a_m)_{\sf st}+\ldots+(a_1)_{\sf st}(a_2)_{\sf st}\cdots(a_{m-1})_{\sf st}(a_m)_{\sf in}\big)\epsilon$ for any $a_i=(a_i)_{\sf st}+(a_i)_{\sf in}\epsilon\in \widehat{\mathbb{C}}$ with $i=1,2,\ldots,m$, by Definition \ref{Def-Det}, it holds that
$$
\begin{array}{l}
{\rm det}[A]\\
\displaystyle=\mathop{\sum}\limits_{(j_1,j_2,\ldots,j_m)}(-1)^{\tau(j_1,j_2,\ldots,j_m)}(a_{1j_1})_{\sf st}(a_{2j_2})_{\sf st}\cdots (a_{mj_m})_{\sf st}\\
~~~\displaystyle+\left(\mathop{\sum}\limits_{(j_1,j_2,\ldots,j_m)}(-1)^{\tau(j_1,j_2,\ldots,j_m)}\sum_{i=1}^m(a_{1j_1})_{\sf st}\cdots(a_{i-1j_{i-1}})_{\sf st}(a_{ij_i})_{\sf in}(a_{i+1j_{i+1}})_{\sf st}\cdots (a_{mj_m})_{\sf st}\right)\epsilon.
\end{array}
$$
From the definition of the determinant of $m\times m$ complex matrices, we know that for every $i=1,2,\ldots,m$,
$$
{\rm det}[A(i)]=\mathop{\sum}\limits_{(j_1,j_2,\ldots,j_m)}(-1)^{\tau(j_1,j_2,\ldots,j_m)}(a_{1j_1})_{\sf st}\cdots(a_{i-1j_{i-1}})_{\sf st}(a_{ij_i})_{\sf in}(a_{i+1j_{i+1}})_{\sf st}\cdots (a_{mj_m})_{\sf st}.
$$
Consequently, we obtain the desired result and complete the proof.
\end{proof}

From (\ref{Det-A}), we know that, $A$ is invertible if and only if ${\rm det}(A_{\rm st})\neq 0$ (i.e., ${\rm det}(A)$ is appreciable), which is equivalent to that $A_{\rm st}$ is invertible.

\begin{Prop}\label{Prop-2}
Let $A\in \widehat{\mathbb{C}}^{m\times m}$, $B\in \widehat{\mathbb{C}}^{n\times n}$, $C\in \widehat{\mathbb{C}}^{n\times m}$ and $D\in \widehat{\mathbb{C}}^{m\times n}$. It holds that
$${\rm det}\left(\begin{array}{cc}
A&O\\
C&B
\end{array}
\right)={\rm det}[A]{\rm det}[B]~~~{\rm and}~~~{\rm det}\left(\begin{array}{cc}
A&D\\
O&B
\end{array}
\right)={\rm det}[A]{\rm det}[B].$$
\end{Prop}
\begin{proof}
We only prove the first formula. The second formula can be proved similarly. Write $A=A_{\sf st}+A_{\sf in}\epsilon$ with $A_{\sf st},A_{\sf in}\in \mathbb{C}^{m\times m}$, $B=B_{\sf st}+B_{\sf in}\epsilon$ with $B_{\sf st},B_{\sf in}\in \mathbb{C}^{n\times n}$, and $C=C_{\sf st}+C_{\sf in}\epsilon$ with $C_{\sf st},C_{\sf in}\in \mathbb{C}^{m\times n}$. Write $G=\left(\begin{array}{cc}
A&O\\
C&B
\end{array}
\right)$. Then $G=G_{\sf st}+G_{\sf in}\epsilon$ with $G_{\sf st}=\left(\begin{array}{cc}
A_{\sf st}&O\\
C_{\sf st}&B_{\sf st}
\end{array}
\right),G_{\sf in}=\left(\begin{array}{cc}
A_{\sf in}&O\\
C_{\sf in}&B_{\sf in}
\end{array}
\right)\in \mathbb{C}^{(m+n)\times (m+n)}$. By Proposition \ref{A-det-Prop-1}, we have
$$
{\rm det}[G]={\rm det}[G_{\sf st}]+\left(\sum_{i=1}^{m+n}{\rm det}[G(i)]\right)\epsilon,
$$
where
$$
G(i)=\left(
\begin{array}{cc}
A(i)&O\\
C_{\sf st}&B_{\sf st}
\end{array}
\right),~~i=1,2,\ldots,m$$
and
$$G(i)=\left(
\begin{array}{cc}
A_{\sf st}&O\\
C(i-m)&B(i-m)
\end{array}
\right),~~i=m+1,m+2,\ldots,m+n.
$$
By well-known matrix theory, we know ${\rm det}[G_{\sf st}]={\rm det}[A_{\sf st}]{\rm det}[B_{\sf st}]$, ${\rm det}[G(i)]={\rm det}[A(i)]{\rm det}[B_{\sf st}]$ for $i=1,2,\ldots,m$
 and ${\rm det}[G(i)]={\rm det}[A_{\sf st}]{\rm det}[B(i-m)]$ for $i=m+1,\ldots,m+n$. Consequently, we have
 $$
 \begin{array}{lll}
 {\rm det}[G]&=&\displaystyle{\rm det}[A_{\sf st}]{\rm det}[B_{\sf st}]+\left({\rm det}[B_{\sf st}]\sum_{i=1}^m{\rm det}[A(i)]+{\rm det}[B_{\sf st}]\sum_{i=1}^{n}{\rm det}[B(i)]\right)\epsilon\\
 &=&\displaystyle\left({\rm det}[A_{\sf st}]+\left(\sum_{i=1}^m{\rm det}[A(i)]\right)\epsilon\right)\left({\rm det}[B_{\sf st}]+\left(\sum_{i=1}^n{\rm det}[B(i)]\right)\epsilon\right)\\
 &=&{\rm det}[A]{\rm det}[B].
 \end{array}
 $$
 We complete the proof.
\end{proof}

\begin{Prop}\label{AB-Product}
Let $C=AB$, where $A,B\in \widehat{\mathbb{C}}^{m\times m}$. It holds that ${\rm det}[C]={\rm det}[A]{\rm det}[B]$.
\end{Prop}
\begin{proof}
Let $D=\left(\begin{array}{cc}
A&O\\
I&B\end{array}\right)$. By Proposition \ref{Prop-2}, it holds that ${\rm det}[D]={\rm det}[A]{\rm det}[B]$. Moreover, by (d) in Property \ref{Five-Prop-Det}, it is easy to see that ${\rm det}[D]={\rm det}[F]$ where $F=\left(\begin{array}{cc}
O&-C\\
I&O\end{array}\right)$, which can be obtained through elementary row transformations similar to complex determinants. On the other hand, by the definition of ${\rm det}[\cdot]$, we know that
$$
{\rm det}[F]=\mathop{\sum}\limits_{(j_1,j_2,\ldots,j_m,j_{m+1},\ldots,j_{2m})}(-1)^{\tau(j_1,j_2,\ldots,j_m,j_{m+1},\ldots,j_{2m})}f_{1j_1}\cdots f_{mj_m} f_{m+1j_{m+1}}\cdots f_{2mj_{2m}}.
$$
From the special structure of $F$, we further know
$$
{\rm det}[F]=(-1)^m\mathop{\sum}\limits_{(j_1,j_2,\ldots,j_m),m+1\leq j_1,j_2,\ldots,j_m\leq 2m}(-1)^{\tau(j_1,j_2,\ldots,j_m,1,2,\ldots,m)}c_{1j_1}c_{2j_2}\cdots c_{mj_{m}}.
$$
Since $m+1\leq j_1,j_2,\ldots,j_m\leq 2m$, it is obvious that $\tau(j_1,j_2,\ldots,j_m,1,2,\ldots,m)=m^2+\tau(j_1,j_2,\ldots,j_m,1,2,\ldots,m)$, which implies
 that
$$
\begin{array}{lll}
{\rm det}[F]&=&\displaystyle(-1)^{m+m^2}\mathop{\sum}\limits_{(j_1,j_2,\ldots,j_m)}(-1)^{\tau(j_1,j_2,\ldots,j_m)}c_{1j_1}c_{2j_2}\cdots c_{mj_{m}}\\
&=&\displaystyle\mathop{\sum}\limits_{(j_1,j_2,\ldots,j_m)}(-1)^{\tau(j_1,j_2,\ldots,j_m)}c_{1j_1}c_{2j_2}\cdots c_{mj_{m}}\\
&=&{\rm det}[C],
\end{array}
$$
where the second equality is due to the fact that $m^2+m$ is always even for any positive integer $m$.
Therefore, we obtain ${\rm det}[C]={\rm det}[A]{\rm det}[B]$ and complete the proof.
\end{proof}

When $A\in\widehat{\mathbb{C}}^{m\times m}$ is invertible, by Proposition \ref{AB-Product}, it holds that ${\rm det}[A]{\rm det}[A^{-1}]=1$ from the fact $AA^{-1}=I$, which implies ${\rm det}[A^{-1}]=({\rm det}[A])^{-1}$. Moreover, by (\ref{e2}) and (\ref{Det-A}), we have
$$
{\rm det}[A^{-1}]=({\rm det}[A_{\sf st}])^{-1}-({\rm det}[A_{\sf st}])^{-2}\left(\sum_{i=1}^m{\rm det}[A(i)]\right)\epsilon.
$$
In particular, if $A$ is unitary, i.e., $A^{-1}=A^*(=\bar A^\top)$, then ${\rm det}[A]=({\rm det}[\bar A^\top])^{-1}=({\rm det}[\bar A])^{-1}$, which implies ${\rm det}[A]{\rm det}[A^*]={\rm det}[A]\overline{{\rm det}[A]}=1$, since ${\rm det}[\bar A]=\overline{{\rm det}[A]}$.

\begin{Prop}\label{ABCD} Let $A\in \widehat{\mathbb{C}}^{m\times m}$, $B\in \widehat{\mathbb{C}}^{m\times n}$, $C\in \widehat{\mathbb{C}}^{n\times m}$ and $D\in \widehat{\mathbb{C}}^{n\times n}$. If $A$ is invertible, then it holds that
$$
{\rm det}\left(\begin{array}{cc}
A&B\\
C&D\\
\end{array}\right)={\rm det}[A]{\rm det}[D-CA^{-1}B].
$$
\end{Prop}
\begin{proof}
Since $$\left(
\begin{array}{cc}
A&B\\
C&D\\
\end{array}
\right)=\left(
\begin{array}{cc}
I&O\\
CA^{-1}&I\\
\end{array}
\right)\left(
\begin{array}{cc}
A&O\\
O&D-CA^{-1}B\\
\end{array}
\right)\left(
\begin{array}{cc}
I&A^{-1}B\\
O&I\\
\end{array}
\right),
$$
the desired formula follows from Proposition \ref{AB-Product} and Proposition \ref{Prop-2}.
\end{proof}

\section{Eigenvalues and characteristic roots of dual number matrices}\label{Eigen-DNMat}

Similar to complex matrices, the eigenvalues and eigenvectors of dual complex matrices were introduced, which have applications in brain science \cite{QACLL22,WDW24}. For $A\in \widehat{\mathbb{C}}^{m\times m}$, if there are $\lambda\in \widehat{\mathbb{C}}$, and ${\bf x}\in \widehat{\mathbb{C}}^m$ with  ${\bf x}$ being appreciable, such that $A{\bf x} = \lambda{\bf x}$, then we say that $\lambda$ is an eigenvalue of $A$, with ${\bf x}$ as an associated eigenvector. 
We have the following propositions and theorem, which can be found in \cite{QC23}.

\begin{Prop}\label{A-similar-B} Let $A, B\in \widehat{\mathbb{C}}^{m\times m}$. If $A\sim B$, i.e., $A = P^{-1}BP$ for some invertible matrix $P\in \widehat{\mathbb{C}}^{m\times m}$, and $\lambda\in \widehat{\mathbb{C}}$ is an eigenvalue of $A$ with an
eigenvector ${\bf x}\in \widehat{\mathbb{C}}^m$. Then $\lambda$ is an eigenvalue of $B$ with an eigenvector $P{\bf x}$.
\end{Prop}

\begin{Prop}
Let $A=A_{\sf st}+A_{\sf in}\epsilon\in \widehat{\mathbb{C}}^{m\times m}$, and $\lambda_{\sf st}\in \mathbb{C}$ be an  eigenvalue
of $A_{\sf st}$. Then there exist $\lambda_{\sf in}\in \mathbb{C}$ and ${\bf x}_{\sf st}, {\bf x}_{\sf in}\in \mathbb{C}$
such that $\lambda=\lambda_{\sf st}+\lambda_{\sf in}\epsilon$ is an eigenvalue of $A$ with an eigenvector ${\bf x}={\bf x}_{\sf st}+{\bf x}_{\sf in}\epsilon$, if and only if ${\bf x}_{\sf st}$ is an eigenvector of $A_{\sf st}$, and
\begin{equation}
A_{\sf st}{\bf x}_{\sf st}\in {\rm Span}(D),\end{equation}
where $D=(A_{\sf st}-\lambda_{\sf st}I_m,{\bf x}_{\sf st})$. If furthermore we
have ${\bf x}_{\sf st}\not\in {\rm Span}(A_{\sf st}-\lambda_{\sf st}I_m)$,
then $\lambda_{\sf in}$, hence $\lambda$, is unique with such ${\bf x}_{\sf st}$. Otherwise, $\lambda_{\sf in}$ can be any complex number.
\end{Prop}

\begin{Thm} \label{Diagable} Let $A=A_{\sf st}+A_{\sf in}\epsilon\in \widehat{\mathbb{C}}^{m\times m}$. Suppose that $A_{\sf st}$ is diagonalizable, i.e.,
$A_{\sf st}=QB_{\sf st}Q^{-1}$ for some invertible matrix $Q\in \mathbb{C}^{m\times m}$, where $B_{\sf st}={\rm diag}(\lambda_{1\sf st}I_{m_1},\ldots,\lambda_{k\sf st}I_{m_k})$ and $\lambda_{1\sf st},\ldots,\lambda_{k\sf st}\in \mathbb{C}$ are distinct complex numbers. 
Then, $A = P^{-1}J_0P$ for some invertible matrix
$P\in \widehat{\mathbb{C}}^{m\times m}$. Here, $J_0$ is the Jordan form of $A$ as expressed by $
J_0={\rm diag}\big(\lambda_{1\sf st}I_{m_1}+J_{1}\epsilon,\lambda_{2\sf st}I_{m_2}+J_{2}\epsilon,\ldots,\lambda_{k\sf st}I_{m_k}+J_{k}\epsilon,\big)
$, where $J_{i}={\rm diag}\big(J_{i1}(\lambda_{i1\sf in}),J_{i2}(\lambda_{i2\sf in}),\ldots,J_{it_i}(\lambda_{it_i\sf in})\big)$ and
$$
J_{ij}(\lambda_{ij\sf in})=\left(
\begin{array}{ccccc}
\lambda_{ij\sf in}&1&\cdots&0&0\\
0&\lambda_{ij\sf in}&\cdots&0&0\\
\vdots&\vdots&\ddots&\vdots&\vdots\\
0&0&\cdots&\lambda_{ij\sf in}&1\\
0&0&\cdots&0&\lambda_{ij\sf in}
\end{array}
\right)_{m_{ij}\times m_{ij}},
$$
with $\sum_{j=1}^{t_i}m_{ij}=m_i$ for $i=1,\ldots,k$.
In fact, the matrix $A$ has $\sum_{i=1}^kt_i$ distinct eigenvalues $\lambda_{i\sf st}+\lambda_{ij\sf in}\epsilon$ for $i=1,2,\ldots,k$ and $j=1,2,\ldots,t_i$. 
\end{Thm}

From Theorem 4.1 in \cite{QL21}, we have the following eigenvalue decomposition theorem of dual complex Hermitian matrices.
\begin{Thm}\label{HUDec} Let $A=A_{\sf st}+A_{\sf in}\epsilon\in \widehat{\mathbb{C}}^{m\times m}$ be Hermitian. Then there are unitary matrix $U\in \widehat{\mathbb{C}}^{m\times m}$ and a diagonal matrix $\Sigma\in \widehat{\mathbb{C}}^{m\times m}$ such that $A=U\Sigma U^*$, where
\begin{equation}\label{Sigmn}\Sigma := {\rm diag} (\mu_1+\mu_{1,1}\epsilon,\ldots,\mu_1+\mu_{1,k_1}\epsilon, \mu_2+\mu_{2,1}\epsilon,\ldots,\mu_r+\mu_{r,k_r}\epsilon),
\end{equation}
where $\mu_1>\mu_2>\ldots>\mu_r$ are real numbers, $\mu_i$ is a $k_i$-multiple eigenvalue of $A_{\sf st}$, $\mu_{i,1}\geq \mu_{i,2}\geq\ldots\geq\mu_{i,k_i}$ are also real numbers. Counting possible multiplicities $\mu_{i,j}$, the form $\Sigma$ is unique.
\end{Thm}

From Theorem \ref{HUDec}, we know that the diagonal elements $\mu_i+\mu_{i,j}\epsilon~(i=1,2,\ldots r,~j=1,2,\ldots,k_i)$ in $\Sigma$ are exactly all eigenvalues of $A$. It is obvious that $k_1+k_2+\ldots+k_r=m$. For the sake of simplicity, we write the $m$ eigenvalues of the dual complex Hermitian  matrix $A\in \widehat{\mathbb{C}}^{m\times m}$ as $\lambda_1\geq\lambda_2\geq\ldots\geq\lambda_m$ in descending order. By the definition of positive (semi-)definiteness  and Theorem \ref{HUDec}, we have the following proposition.

\begin{Prop} \label{Positive-SemiDef} Let $A\in \widehat{\mathbb{C}}^{m\times m}$ be Hermitian. We have the following conclusions.


(a) If $A$ is positive semidefinite, then $PAP^*$ is positive semidefinite  for any $P\in \widehat{\mathbb{C}}^{k\times m}$.

(b) $A$ is positive semidefinite (definite) if and only if its eigenvalues are nonnegative (positive).
\end{Prop}

Now we introduce the conception of the characteristic polynomial of dual complex matrices. For given $A=A_{\sf st}+A_{\sf in}\epsilon\in \widehat{\mathbb{C}}^{m\times m}$, define $f_A:\widehat{\mathbb{C}}\rightarrow\widehat{\mathbb{C}}$ by
$$
f_A(\lambda)={\rm det}[\lambda I-A],
$$
where $\lambda=\lambda_{\sf st}+\lambda_{\sf in}\epsilon\in \widehat{\mathbb{C}}$. We call the $f_A$ above the characteristic polynomial of $A$, and call $\lambda=\lambda_{\sf st}+\lambda_{\sf in}\epsilon\in \widehat{\mathbb{C}}$ satisfying $f_A(\lambda)=0$ the characteristic root of $A$. By Proposition \ref{A-det-Prop-1}, we know that
\begin{equation}\label{Det-Lambda}
f_A(\lambda)=g_A(\lambda_{\sf st})+\left(\sum_{i=1}^m{\rm det}[\tilde{A}(i,\lambda)]\right)\epsilon,
\end{equation}
where $g_A(\lambda_{\sf st}):={\rm det}[\lambda_{\sf st} I-A_{\sf st}]$ and
$$
\begin{array}{l}
\tilde{A}(i,\lambda)\\
=\left(\begin{array}{ccccccc}
\lambda_{\sf st}-(a_{11})_{\sf st}&\cdots&-(a_{1i-1})_{\sf st}&-(a_{1i})_{\sf st}&-(a_{1i+1})_{\sf st}&\cdots&-(a_{1m})_{\sf st}\\
\vdots&\ddots&\vdots&\vdots&\vdots&\ddots&\vdots\\
-(a_{i-11})_{\sf st}&\cdots&\lambda_{\sf st}-(a_{i-1i-1})_{\sf st}&-(a_{i-1i})_{\sf st}&-(a_{i-1i+1})_{\sf st}&\cdots&-(a_{i-1m})_{\sf st}\\
-(a_{i1})_{\sf in}&\cdots&-(a_{ii-1})_{\sf in}&\lambda_{\sf in}-(a_{ii})_{\sf in}&-(a_{ii+1})_{\sf in}&\cdots&-(a_{im})_{\sf in}\\
-(a_{i+11})_{\sf st}&\cdots&-(a_{i+1i-1})_{\sf st}&-(a_{i+1i})_{\sf st}&\lambda_{\sf st}-(a_{i+1i+1})_{\sf st}&\cdots&-(a_{i+1m})_{\sf st}\\
\vdots&\ddots&\vdots&\vdots&\vdots&\ddots&\vdots\\
-(a_{m1})_{\sf st}&\cdots&-(a_{mi-1})_{\sf st}&-(a_{mi})_{\sf st}&-(a_{mi+1})_{\sf st}&\cdots&\lambda_{\sf st}-(a_{mm})_{\sf st}\\
\end{array}
\right)
\end{array}
$$
for $i=1,2,\ldots,m$. It is obvious that
\begin{equation}\label{Ai-tilde}
{\rm det}\big[\tilde{A}(i,\lambda)\big]=\lambda_{\sf in}{\rm det}\big[{\lambda_{\sf st}I_{m-1}-(A_{ii})_{\sf st}}\big]+{\rm det}\big[\bar A(i)\big],
\end{equation}
where $A_{ii}$ is the complementary submatrix of $A$ with respect to $a_{ii}$ and
$$
\begin{array}{l}
\bar{A}(i)\\
=\left(\begin{array}{ccccccc}
\lambda_{\sf st}-(a_{11})_{\sf st}&\cdots&-(a_{1i-1})_{\sf st}&-(a_{1i})_{\sf st}&-(a_{1i+1})_{\sf st}&\cdots&-(a_{1m})_{\sf st}\\
\vdots&\ddots&\vdots&\vdots&\vdots&\ddots&\vdots\\
-(a_{i-11})_{\sf st}&\cdots&\lambda_{\sf st}-(a_{i-1i-1})_{\sf st}&-(a_{i-1i})_{\sf st}&-(a_{i-1i+1})_{\sf st}&\cdots&-(a_{i-1m})_{\sf st}\\
-(a_{i1})_{\sf in}&\cdots&-(a_{ii-1})_{\sf in}&-(a_{ii})_{\sf in}&-(a_{ii+1})_{\sf in}&\cdots&-(a_{im})_{\sf in}\\
-(a_{i+11})_{\sf st}&\cdots&-(a_{i+1i-1})_{\sf st}&-(a_{i+1i})_{\sf st}&\lambda_{\sf st}-(a_{i+1i+1})_{\sf st}&\cdots&-(a_{i+1m})_{\sf st}\\
\vdots&\ddots&\vdots&\vdots&\vdots&\ddots&\vdots\\
-(a_{m1})_{\sf st}&\cdots&-(a_{mi-1})_{\sf st}&-(a_{mi})_{\sf st}&-(a_{mi+1})_{\sf st}&\cdots&\lambda_{\sf st}-(a_{mm})_{\sf st}\\
\end{array}
\right).
\end{array}
$$
Moreover, it is obvious that $f_A(\lambda)=0$ if and only if
$$
\left\{
\begin{array}{l}
{\rm det}[\lambda_{\sf st} I-A_{\sf st}]=0\\
\displaystyle\sum_{i=1}^m{\rm det}[\tilde{A}(i,\lambda)]=0,
\end{array}
\right.
$$
which is equivalent to, by (\ref{Ai-tilde}), that
\begin{equation}\label{Charact-Root-e1}
\left\{
\begin{array}{l}
{\rm det}[\lambda_{\sf st} I-A_{\sf st}]=0\\
\displaystyle\lambda_{\sf in}\sum_{i=1}^m{\rm det}\big[{\lambda_{\sf st}I_{m-1}-(A_{ii})_{\sf st}}\big]=-\sum_{i=1}^m{\rm det}\big[\bar A(i)\big].
\end{array}
\right.
\end{equation}
It is easy to see that $\frac{dg_A(\lambda_{\sf st})}{d\lambda_{\sf st}}=\sum_{i=1}^m{\rm det}\big[{\lambda_{\sf st}I_{m-1}-(A_{ii})_{\sf st}}\big]$. Hence, (\ref{Charact-Root-e1}) can be written as
\begin{equation}\label{Charact-Root-e2}
\left\{
\begin{array}{l}
g_A(\lambda_{\sf st})=0\\
\displaystyle\lambda_{\sf in}\frac{dg_A(\lambda_{\sf st})}{d\lambda_{\sf st}}+\tau(\lambda_{\sf st})=0,
\end{array}
\right.
\end{equation}
where $\tau(\lambda_{\sf st})=\sum_{i=1}^m{\rm det}\big[\bar A(i)\big]$.

\begin{example}A dual number matrix $A$, which has no characteristic root at all.
Let $A=A_{\sf st} + A_{\sf in}\epsilon$, where
$$
A_{\sf st}=\left(\begin{array}{cc}
1&1\\
0&1
\end{array}\right)~~~{\rm and}~~~A_{\sf in}=\left(\begin{array}{cc}
0&0\\
1&0
\end{array}\right).
$$
Then $f_A(\lambda)={\rm det}[\lambda I-A]=(\lambda_{\sf st}-1)^2+2(\lambda_{\sf st}-1)\lambda_{\sf in}\epsilon-\epsilon$. Consequently, all possible characteristic roots $\lambda=\lambda_{\sf st}+\lambda_{\sf in}\epsilon$ of $A$ must satisfy $\lambda_{\sf st}=1$. Hence, $f_A(\lambda)=-\epsilon\neq 0$.  i.e.,
$A$ has no characteristic root at all.
\end{example}

\begin{Prop}\label{Eig-Char}
Let $A=A_{\sf st}+A_{\sf in}\epsilon\in \widehat{\mathbb{C}}^{m\times m}$, and $\lambda_{\sf st}\in \mathbb{R}$ be an eigenvalue of $A_{\sf st}$. If $\lambda_{\sf st}$ is a single eigenvalue of $A_{\sf st}$, then $\lambda=\lambda_{\sf st}+\lambda_{\sf in}\epsilon$ is a characteristic root of $A$, where $\lambda_{\sf in}$ is uniquely determined by $\lambda_{\sf in}=-\tau(\lambda_{\sf st})/\frac{dg_A(\lambda_{\sf st})}{d\lambda_{\sf st}}$; If $\lambda_{\sf st}$ is an eigenvalue of $A_{\sf st}$ with algebraic multiplicity $k\geq 2$, then for any $b\in \mathbb{R}$, $\lambda=\lambda_{\sf st}+b\epsilon$ is  characteristic root of $A$.
\end{Prop}

\begin{proof}
If $\lambda_{\sf st}$ is a single eigenvalue of $A_{\sf st}$, then $\frac{dg_A(\lambda_{\sf st})}{d\lambda_{\sf st}}\neq0$. Consequently, the first conclusion comes from (\ref{Charact-Root-e2}). If $\lambda_{\sf st}$ ia an eigenvalue of $A_{\sf st}$ with algebraic multiplicity $k\geq 2$, then we know that $\frac{dg_A(\lambda_{\sf st})}{d\lambda_{\sf st}}=0$. Moreover, under the given condition on $A_{\sf st}$, we know that ${\rm rank}(\lambda_{\sf st}I-A_{\sf st})=m-k\leq m-2$, which implies ${\rm rank}(\bar A(i))\leq m-1$ for every $i\in [m]$. Hence,
${\rm det}\big[\bar A(i)\big]=0$ for every $i\in [m]$. Consequently, the second expression in (\ref{Charact-Root-e2}) holds for any $\lambda_{\sf in}=b\in \mathbb{R}$.
\end{proof}

 For given $A\in \widehat{\mathbb{C}}^{m\times m}$ with $A_{\sf st}$ being diagonalizable, by Theorem \ref{Diagable}, we have $f_A(\lambda)={\rm det}[\lambda I-A]={\rm det}[\lambda I-J_0]=\prod_{i=1}^k\big(\prod_{j=1}^{t_i}(\lambda-\lambda_{i\sf st}-\lambda_{ij\sf in}\epsilon)\big)$, which mens that the eigenvalues of $A\in \widehat{\mathbb{C}}^{m\times m}$ are necessarily the characteristic roots of $A$. In particular, when $A\in \widehat{\mathbb{C}}^{m\times m}$ is Hermitian, due to the fact that $A_{\sf st}$ is diagonalizable, we know that the eigenvalue set $\{\lambda_1,\lambda_2,\ldots,\lambda_m\}$ of $A$ must be a subset of the set composed of characteristic roots of $A$. The following theorem indicates that the similar conclusion also holds for general dual complex matrices.

\begin{Thm}\label{Eig-Root}
Let $A=A_{\sf st}+A_{\sf in}\epsilon\in \widehat{\mathbb{C}}^{m\times m}$. The eigenvalues of $A$ must be the characteristic roots of $A$.
\end{Thm}
\begin{proof}
Take any eigenvalue $\bar{\lambda}$ of $A$. Denote by $\mathbb{U}$ the set of all eigenvectors of $A$ with respect to $\bar{\lambda}$. It is obvious that $\mathbb{U}$ is a subspace in $\widehat{\mathbb{C}}^m$ with $s:={\rm dim}(\mathbb{U})\geq 1$. Let $\{{\bf u}_1,\ldots, {\bf u}_s\}$ be a basis of $\mathbb{U}$. Denote $U_1=({\bf u}_1,\ldots, {\bf u}_s)$. It is clear that $U_1\in \widehat{\mathbb{C}}^{m\times s}$. By Proposition 3.4 in \cite{LHQ22}, there exists a dual number matrix $U_2\in \widehat{\mathbb{C}}^{m\times (m-s)}$ such that $U=(U_1, U_2)$ is invertible. It is obvious that $U^{-1}AU=B$, where
$$
B=\left(
\begin{array}{cc}
\bar{\lambda} I_s&B_{12}\\
O&B_{22}
\end{array}\right)~{\rm with}~~B_{12}\in \widehat{\mathbb{C}}^{s\times (m-s)}~~{\rm and}~~ B_{22}\in \widehat{\mathbb{C}}^{(m-s)\times (m-s)}.
$$
By Proposition \ref{AB-Product} and Proposition \ref{Prop-2}, we know $$f_A(\lambda)={\rm det}[\lambda I-A]={\rm det}[\lambda I-B]=(\lambda-\bar{\lambda})^s{\rm det}[\lambda I_{m-s}-B_{22}].$$
Consequently, we obtain the desired conclusion and complete the proof.
\end{proof}
 From Propositions \ref{Eig-Char} and Theorem \ref{Eig-Root}, we know that, for given $A\in \widehat{\mathbb{C}}^{m\times m}$ with $A_{\rm st}$ having $m$ distinct eigenvalues, then the eigenvalues of $A$ are exactly the same as its characteristic roots. The following example shows that a characteristic root of a general dual number matrix $A$ may not necessarily be the eigenvalue of $A$.
 \begin{example}
 Consider $A=A_{\sf st}+A_{\sf in}\epsilon\in \widehat{\mathbb{C}}^{3\times 3}$ with
 $$
 A_{\sf st}=\left(\begin{array}{ccc}
 0&1&1\\
 1&0&1\\
 1&1&0
 \end{array}\right)~~~{\rm and}~~~A_{\sf in}=\left(\begin{array}{ccc}
 1&0&0\\
 0&1&0\\
 0&0&1
 \end{array}\right).
 $$It is easy to see that $f_A(\lambda)=(\lambda+1-\epsilon)^2(\lambda-2-\epsilon)$, which implies that all $\lambda=-1+b\epsilon$ are characteristic roots of $A$ for any $b\in \mathbb{R}$. It is obvious that an eigenvector of $A_{\sf st}$ with respect to the eigenvalue $\lambda_{\sf st}=-1$ is $(-1,1,0)^\top$. We claim that $\lambda=-1+b\epsilon$ with $b\neq 1$ is not an eigenvalue of $A$. In fact, if $\lambda=-1+b\epsilon$ is an eigenvalue of $A$, then the equation $A{\bf x}=\lambda {\bf x}$ must have a solution ${\bf x}=(-1,1,0)^\top+{\bf y}\epsilon$, which means $(A_{\sf st}-\lambda_{\sf st}I){\bf y}=(-b+1,b-1,0)^\top$ must have a solution. By the classical linear equation theory, we know $b=1$, which is a contradiction.
 \end{example}

By Theorem \ref{HUDec} and Proposition \ref{AB-Product}, we know that, when $A\in \widehat{\mathbb{C}}^{m\times m}$ is Hermitian , we have ${\rm det}[A]={\rm det}[U\Sigma U^*]={\rm det}[\Sigma]=\prod_{i=1}^m\lambda_i$, where the second equality is due to ${\rm det}[U]{\rm det}[U^*]=1$ which comes from the unitarity of $U$. However, for a general $A\in \widehat{\mathbb{C}}^{m\times m}$, this conclusion does not hold.

\begin{example}\label{Ex-1}
 Consider $A=A_{\sf st}+A_{\sf in}\epsilon$ with
 $$
 A_{\sf st}=\left(\begin{array}{ccc}
 -1&2&2\\
 3&-1&1\\
 2&2&-1
 \end{array}\right)~~~{\rm and}~~~A_{\sf in}=\left(\begin{array}{ccc}
 1&0&-1\\
 -2&0&3\\
 2&1&0
 \end{array}\right).
 $$
 It is easy to see that $g_A(\lambda_{\sf st})={\rm det}[\lambda_{\sf st}I-A_{\sf st}]=(\lambda_{\sf st}-3)(\lambda_{\sf st}+3)^2$. Moreover, it is easy to verify that $\lambda_1=3+(7/6)\epsilon$ is an eigenvalue of $A$ and the corresponding eigenvector is $\bar{\bf x}=(1,1,1)^\top-(1/2,5/12,0)^\top \epsilon$. However, by a simple computation, we know
 $$
 {\rm det}[\tilde{A}(1,\lambda_1)]=49/3,~~{\rm det}[\tilde{A}(2,\lambda_1)]=2,~~{\rm det}[\tilde{A}(3,\lambda_1)]=-55/3,
 $$
 which implies $f_A(\lambda_1)=0$ by (\ref{Det-Lambda}). Hence $\lambda_1$ is a characteristic root of $A$. Moreover, it is also easy to verify that, for every $a\in \mathbb{R}$, $\lambda_2=-3+a\epsilon$ is also an eigenvalue of $A$ and the corresponding eigenvector is $\hat{\bf x}=(1,-2,1)^\top+(-3a-1,(7/2)a+1,0)^\top \epsilon$. In this case, we call $\lambda_2$ an eigenvalue of $A$ with appreciable algebraic multiplicity $2$. By a simple computation, we know
 $$
 {\rm det}[\tilde{A}(1,\lambda_2)]=2a,~~{\rm det}[\tilde{A}(2,\lambda_2)]=0,~~{\rm det}[\tilde{A}(3,\lambda_2)]=-2a,
 $$
 which implies $f_A(\lambda_2)=0$ by (\ref{Det-Lambda}). Hence, $\lambda_2$ is a characteristic root of $A$ for every $a\in \mathbb{R}$. On the other hand, we have ${\rm det}[A]=27+12\epsilon$, and $\lambda_1\lambda_2\lambda_3=27+(21/2-9a-9b)\epsilon$, where $\lambda_3=-3+b\epsilon$ for $b\in \mathbb{R}$. Hence ${\rm det}[A]\neq\lambda_1\lambda_2\lambda_3$.
 \end{example}


\begin{Prop} For any given $A,B\in \widehat{\mathbb{C}}^{m\times m}$, it holds that $f_{AB}(\lambda)=f_{BA}(\lambda)$.
\end{Prop}

\begin{proof}
By Theorem \ref{SVD-DQM}, there exist dual complex unitary matrices $U,V\in \widehat{\mathbb{C}}^{m\times m}$, such that
\begin{equation}\label{mmSVDDQMEQ}
A=U\left(\begin{array}{cc}\Sigma_s&O\\
O&O
\end{array} \right)_{m\times m}V^*,
\end{equation}
where $\Sigma_s={\rm diag} (\sigma_1,\ldots, \sigma_r,\sigma_{r+1},\ldots,\sigma_{s})\in \widehat{\mathbb{C}}^{s\times s}$, with $\sigma_1\geq \ldots\geq\sigma_r$ being positive appreciable dual numbers, and $\sigma_{r+1}\geq \ldots\geq\sigma_s$ being positive infinitesimal dual numbers. Denote
$$
R={\rm diag}(\sqrt{\sigma_1},\ldots,\sqrt{\sigma_r},\sqrt{\sigma_{r+1\sf in}},\ldots,\sqrt{\sigma_{s\sf in}},I_{m-s})
$$
and $P=UR$ and Q=$RV^*$. It is obvious that $P, Q\in \widehat{\mathbb{C}}^{m\times m}$ are invertible. Moreover, it is easy to see that $A=PDQ$, where $D={\rm diag}(I_r,I_{s-r}\varepsilon,O_{(m-s)\times (m-s)})$. Consequently, it holds that
$$
AB=PDQBPP^{-1}=PDGP^{-1}~~~{\rm and}~~~BA=Q^{-1}QBPDQ=Q^{-1}GDQ,
$$
where $G:=QBP=\left(
\begin{array}{ccc}
G_{11}&G_{12}&G_{13}\\
G_{21}&G_{22}&G_{23}\\
G_{31}&G_{32}&G_{33}\\
\end{array}
\right)$. By the definition of the characteristic polynomial, we have $f_{AB}(\lambda)=f_{DG}(\lambda)$ and $f_{BA}(\lambda)=f_{GD}(\lambda)$. Moreover, it is not difficult to know that
$$
f_{DG}(\lambda)={\rm det}[\left(\begin{array}{cc}
\lambda I_r-G_{11}&-G_{12}\\
-G_{21}\epsilon&\lambda I_{s-r}-G_{22}\epsilon
\end{array}\right)]\lambda^{m-s},
$$
which implies, together with Propositions \ref{AB-Product} and \ref{ABCD}, that
$$
f_{DG}(\lambda)={\rm det}[\left(\begin{array}{cc}
\lambda I_r-G_{11}&O\\
O&\lambda I_{s-r}-G_{22}\epsilon-G_{21}(\lambda I_r-G_{11})^{-1}G_{12}\epsilon
\end{array}\right)]\lambda^{m-s},
$$
where $\lambda$ can be selected to be large enough such that $(\lambda I_r-G_{11})^{-1}$ exists.
Similarly, we can obtain
$$
f_{GD}(\lambda)={\rm det}[\left(\begin{array}{cc}
\lambda I_r-G_{11}&O\\
O&\lambda I_{s-r}-G_{22}\epsilon-G_{21}(\lambda I_r-G_{11})^{-1}G_{12}\epsilon
\end{array}\right)]\lambda^{m-s}.
$$
Hence, we know $f_{AB}(\lambda)=f_{BA}(\lambda)$ and complete the proof.
\end{proof}



\begin{Thm}\label{Det-A+B} Let $A,B\in \widehat{\mathbb{C}}^{m\times m}$ be Hermitian. If $A, B$ are positive semidefinite, then it holds that ${\rm det}[A+B]\geq {\rm det}[A]+{\rm det}[B]$.
\end{Thm}

\begin{proof} It is easy to see that ${\rm det}[A], {\rm det}[B],{\rm det}[A+B]\in \widehat{\mathbb{R}}$, since $A,B$ are Hermitian. By Theorem \ref{HUDec}, there are unitary matrix $U\in \widehat{\mathbb{C}}^{m\times m}$ and a diagonal matrix $\Sigma\in \widehat{\mathbb{C}}^{m\times m}$ such that $B=U\Sigma U^*$, where $\Sigma := {\rm diag} (\lambda_1,\lambda_2,\ldots,\lambda_m)$. Notice that $\lambda_i\geq 0$ for any $i=1,2,\ldots,m$, which from the given condition that $B$ is positive semidefinite. Since $U$ is unitary, we have ${\rm det}[A+B]= {\rm det}[C+\Sigma]$, where $C=U^*AU$. Moreover, since $\Sigma$ is diagonal, by (\ref{A+B}), it holds that
\begin{equation}\label{ABC+}
\begin{array}{l}
{\rm det}[C+\Sigma]\\
={\rm det}[C]+\mathop{\sum}\limits_{1\leq i_1\leq m}\lambda_{i_1}{\rm det}\big[\hat{A}_C\left(\begin{array}{c}i_1\\i_1\end{array}\right)\big]+\mathop{\sum}\limits_{1\leq i_1<i_2\leq m}\lambda_{i_1}\lambda_{i_2}{\rm det}\big[\hat{A}_C\left(\begin{array}{c}i_1i_2\\i_1i_2\end{array}\right)\big]+\ldots\\
~~+\mathop{\sum}\limits_{1\leq i_1<\ldots <i_{m-1}\leq m}\lambda_{i_1}\cdots\lambda_{i_{m-1}}{\rm det}\big[\hat{A}_C\left(\begin{array}{c}i_1\cdots i_{m-1}\\i_1\cdots i_{m-1}\end{array}\right)\big]+\prod_{i=1}^m\lambda_i.
\end{array}
\end{equation}
Since $A$ is positive semidefinite, it is obvious that ${\rm det}\big[\hat{A}_C\left(\begin{array}{c}i_1\cdots i_{k}\\i_1\cdots i_{k}\end{array}\right)\big]\geq 0$ for any $k=1,2,\ldots, m-1$. Consequently, by (\ref{ABC+}) and $\lambda_i\geq 0$ for $i=1,2,\ldots,m$, it holds that $${\rm det}[C+\Sigma]\geq {\rm det}[C]+\prod_{i=1}^m\lambda_i={\rm det}[A]+{\rm det}[B],$$ where the inequality comes from Proposition \ref{P6.5}, and the last equality is due to ${\rm det}[C]={\rm det}[A]$ and  $\prod_{i=1}^m\lambda_i={\rm det}[B]$, which means that the desired result holds, since ${\rm det}[A+B]= {\rm det}[C+\Sigma]$. We complete the proof.
\end{proof}

\begin{Thm}\label{D-PosDet}
Let $A\in \widehat{\mathbb{C}}^{m\times m}$, $B\in \widehat{\mathbb{C}}^{n\times n}$ and $C\in \widehat{\mathbb{C}}^{m\times n}$. If $
D=\left(
\begin{array}{cc}
A&C\\
C^\top&B\\
\end{array}
\right)
$
is positive semidefinite, then we have
${\rm det}[D]\leq {\rm det}[A]{\rm det}[B]$.
\end{Thm}
\begin{proof}
Since $D$ is positive semidefinite, it is obvious that $A,B$ are positive semidefinite, which implies ${\rm det}[A]\geq 0$, ${\rm det}[B]\geq 0$ and ${\rm det}[D]\geq 0$. We only need to prove the inequality for the case $D$ is positive definite. 
 In this case, we know that $A$ is positive definite, which implies that $A^{-1}$ is also positive definite. Consequently, by Proposition \ref{ABCD}, it holds that ${\rm det}[D]= {\rm det}[A]{\rm det}[B-C^\top A^{-1}C]$. Since $C^\top A^{-1}C$ is positive semidefinite, by Theorem \ref{Det-A+B}, we know ${\rm det}[B-C^\top A^{-1}C]\leq {\rm det}[B]$, which implies, together with ${\rm det}[A]>0$, that ${\rm det}[D]\leq {\rm det}[A]{\rm det}[B]$. We complete the proof.
\end{proof}
From Theorem \ref{D-PosDet}, by using mathematical induction, we can easily prove that, if
$$
A=\left(
\begin{array}{cccc}
A_{11}&A_{12}&\cdots&A_{1k}\\
A_{21}&A_{22}&\cdots&A_{2k}\\
\vdots&\vdots&\ddots&\vdots\\
A_{k1}&A_{k2}&\cdots&A_{kk}\\
\end{array}
\right)
$$
is positive semidefinite, where $A_{ii}~(i\in [k])$ are square, then we have ${\rm det}[A]\leq\prod_{i=1}^k{\rm det}[A_{ii}]$. In particular, if $A=(a_{ij})\in \widehat{\mathbb{C}}^{m\times m}$ is positive semidefinite, then ${\rm det}[A]\leq a_{11}a_{22}\cdots a_{mm}$.

\begin{Thm}\label{Cauchy-Schwarz-Inequality}
For any given $A,B\in \widehat{\mathbb{C}}^{m\times n}$, it holds that
$$
\big|{\rm det}[A^* B]\big|^2\leq {\rm det}[A^* A]{\rm det}[B^* B].
$$
\end{Thm}
\begin{proof}
We prove this conclusion for two cases: (i) ${\rm Arank}(A)<n$, and (ii) ${\rm Arank}(A)=n$. When (i) ${\rm Arank}(A)<n$, by Theorem \ref{SVD-DQM}, there exists a dual complex unitary matrix $V\in \widehat{\mathbb{C}}^{n\times n}$, such that $$A^* A=V\left(\begin{array}{cc}\Sigma_r^2&O\\
O&O
\end{array} \right)_{n\times n}V^*,$$
where $\Sigma_r={\rm diag} (\sigma_1,\ldots, \sigma_r,0\ldots,0)$ with $r={\rm Arank}(A)$. Since $r<n$, we know ${\rm det}[A^* A]=0$ from Proposition \ref{AB-Product}. On the other hand, by Theorem 4.8 and Proposition 4.5 in \cite{LHQ22}, it holds that ${\rm Arank}(A^* B)\leq{\rm Arank}(A)<n$. Consequently, by applying Theorem \ref{SVD-DQM} to $A^*B$, we know that ${\rm det}[A^* B]$ is infinitesimal, which implies $\big|{\rm det}[A^* B]\big|^2=0$. Hence, the desired conclusion holds for the case (i).

Now we prove the conclusion for the case (ii). Denote
$$
C=\left(\begin{array}{cc}
A^* A&B^* A\\
A^* B&B^* B
\end{array}\right).
$$
 It is obvious that $C$ is positive semidefinite, which implies ${\rm det}[C]\geq 0$. It is obvious that $A^* A$ is positive definite since ${\rm Arank}(A)=n$, which implies that $A^*A$ is invertible and the dual number ${\rm det}[A^* A]$ is  positive and appreciable. It is easy to verify that
$$
\begin{array}{l}
\left(\begin{array}{cc}
A^* A&O\\
O&B^* B-A^* B(A^* A)^{-1}B^* A
\end{array}\right)
=\left(\begin{array}{cc}
I&O\\
-A^* B(A^* A)^{-1}&I
\end{array}\right)C\left(\begin{array}{cc}
I&O\\
-A^* B(A^* A)^{-1}&I
\end{array}\right)^*.
\end{array}
$$
Consequently, since $C$ is positive semidefinite, by (a) in Proposition \ref{Positive-SemiDef}, we know that $B^* B-A^* B(A^* A)^{-1}B^* A$ is positive semidefinite, which implies, together with (b) in Proposition \ref{Positive-SemiDef} and Proposition \ref{P6.5}, that ${\rm det}[B^* B-A^* B(A^* A)^{-1}B^* A]\geq 0$. Since $B^* B-A^* B(A^* A)^{-1}B^* A$ and $A^* B(A^* A)^{-1}B^* A$ are positive semidefinite, by Proposition \ref{Det-A+B}, we have
$$
\begin{array}{lll}
{\rm det}[B^* B]&\geq &{\rm det}[B^* B-A^* B(A^* A)^{-1}B^* A]
+{\rm det}[A^* B(A^* A)^{-1}B^* A]\\
&\geq& {\rm det}[A^* B(A^* A)^{-1}B^* A]\\
&=&{\rm det}[A^* B]{\rm det}[(A^* A)^{-1}]{\rm det}[B^* A]\\
&=&{\rm det}[A^* B] {\rm det}[\overline{A^* B}]({\rm det}[A^* A])^{-1},
\end{array}
$$
which implies, together with Proposition \ref{P6.5}, that $\big|{\rm det}[A^* B]\big|^2\leq {\rm det}[A^* A]{\rm det}[B^* B]$, since ${\rm det}[\overline{A^* B}]=\overline{{\rm det}[A^* B]}$. We obtain the desired result and complete the proof.
\end{proof}

\begin{Thm}\label{Lemma3}
\cite{LQY22} Let $A\in \widehat{\mathbb{C}}^{m\times m}$ be Hermitian, and $\lambda_1\geq\lambda_2\geq\ldots\geq\lambda_m$ be eigenvalues of $A$. Then for $k=2,3,\ldots,m$, we have
\begin{equation}\label{MLambda_k}
\lambda_k=\min_{B\in \widehat{\mathbb{C}}^{m\times (k-1)}}\max_{{\bf x}\in N(B^*)\backslash[{\bf 0}]}\|{\bf x}\|^{-2}({\bf x}^* A{\bf x}),
\end{equation}
and it attains $\lambda_k$ when $B=[{\bf u}_1,{\bf u}_2,\ldots,{\bf u}_{k-1}]$;
and for $k=1,2,\ldots,m-1$, we have
\begin{equation}\label{MLambda_{n-k}}\lambda_{m-k}=\max_{C\in \widehat{\mathbb{C}}^{m\times k}}\min_{{\bf x}\in N(C^*)\backslash[{\bf 0}]}\|{\bf x}\|^{-2}({\bf x}^* A{\bf x}),
\end{equation}
and it attains the $\lambda_{m-k}$ when $C=[{\bf u}_{m-k+1},{\bf u}_{m-k+2},\ldots,{\bf u}_m]$, where ${\bf u}_i$ is the $i$th column of unitary matrix $U$ in  Theorem \ref{HUDec}. Here, for given $W\in \widehat{\mathbb{C}}^{p\times q}$, $N(W):=\{{\bf z}\in \widehat{\mathbb{C}}^{q}~|~W{\bf z}={\bf 0}\}$. 
\end{Thm}

From Theorem \ref{Lemma3}, we can obtain the following theorem, which is a dual complex matrix version of the well-known Sturm Theorem for complex matrices, and can be proven using a method similar to that used in complex matrix theory.

\begin{Thm}\label{SturmThm}
Let $A\in \widehat{\mathbb{C}}^{m\times m}$ be Hermitian. Then for any $k$-th order principal submatrix $A_k$ of $A$, it holds that
$\lambda_{m-k+i}(A)\leq\lambda_i(A_k)\leq\lambda_i(A)$ for every $i=1,2,\ldots,k$.
\end{Thm}

\begin{Thm}(Bloomfield-Watson inequality)
Let $A\in \widehat{\mathbb{C}}^{m\times m}$ be positive semidefinite, and $k<m$. Then for any $X\in \widehat{\mathbb{C}}^{m\times k}$ satisfying $X^* X=I_k$, it holds that
\begin{equation}\label{XAX-Eig-i}
\prod_{i=1}^k\lambda_{m-k+i}(A)\leq {\rm det}[X^* AX]\leq\prod_{i=1}^k\lambda_{i}(A).
\end{equation}
\end{Thm}
\begin{proof}
It is obvious that $\lambda_i(A)\geq 0$ for any $i=1,2,\ldots,m$. Moreover, by (a) in Proposition \ref{Positive-SemiDef}, we know that $X^* AX$ is positive semidefinite. Since ${\rm det}[X^* AX]=\prod_{i=1}^k\lambda_i(X^* AX)$ and $\lambda_i(X^* AX)\geq 0$, to prove (\ref{XAX-Eig-i}), we only need to prove $\lambda_{m-k+i}(A)\leq\lambda_i(X^* AX)\leq\lambda_{i}(A)$ for every $i=1,2,\ldots,k$. By Proposition \ref{Prop3.4}, there exists a $V \in {\mathbb {Q}}^{m \times (m-k)}$ such that $Y:=(X, V) \in {\mathbb{Q}}^{m \times m}$ is unitary. Denote
$$
\tilde{A}=Y^* AY=\left(\begin{array}{cc}
X^* AX&X^* AV\\
V^* AX&V^* AV
\end{array}\right).
$$
Then by Proposition \ref{A-similar-B}, we have $\lambda_{i}(\tilde{A})=\lambda_{i}(A)$ for $i=1,2,\ldots,m$. Notice that $X^* AX$ is the $k$-th order principal submatrix of $\tilde{A}$. By Theorem \ref{SturmThm}, we know
$$
\lambda_{m-k+i}(A)=\lambda_{m-k+i}(\tilde{A})\leq\lambda_i(X^* AX)\leq\lambda_{i}(\tilde{A})=\lambda_{i}(A)
$$
for every $i=1,2,\ldots,k$. Consequently, by Proposition \ref{P6.5}, we obtain the desired result and complete the proof.
\end{proof}

\section{Quasi-determinant of dual quaternion matrices}\label{DET-DQ}
Inspired by the work in \cite{R14}, let us define the map $\omega:\widehat{\mathbb{Q}}\rightarrow \widehat{\mathbb{C}}^{2\times 2}$ by 
$$
\omega(a)=\left(
\begin{array}{cc}
a_0+a_1{\bf i}&a_2+a_3{\bf i}\\
-a_2+a_3{\bf i}&a_0-a_1{\bf i}
\end{array}
\right),~~~a=a_0+a_1{\bf i}+a_2{\bf j}+a_3{\bf k}\in \widehat{\mathbb{Q}}~{\rm with}~ a_0, a_1, a_2, a_3\in \widehat{\mathbb{R}}. 
$$
Denote by $\mathbb{Z}$ the all $2\times 2$ dual complex matrices of the form as follows
$$
\mathbb{Z}=\left\{\left(\begin{array}{cc}u&v\\
-\bar v&\bar u
\end{array}\right)\in \widehat{\mathbb{C}}^{2\times 2}~|~u,v\in \widehat{\mathbb{C}}\right\}.$$
It is easy to verify that $Z_1+Z_2\in \mathbb{Z}$ and $Z_1Z_2\in \mathbb{Z}$ for any $Z_1,Z_2\in \mathbb{Z}$. Furthermore, from the definition of the map $\omega$, it is easy to verify that $\omega$ is one-to-one and onto map on the set $\mathbb{Z}$.

\begin{Prop}\label{UnitIsomMap} The map $\omega$ is a unital isomorphism of $\widehat{\mathbb{Q}}$ onto the set $\mathbb{Z}$.
\end{Prop}
\begin{proof}
It is easy to see that $\omega(1) = I_2$, $\omega(\bar a)=\omega(a)^*$, and $\omega(a+b)=\omega(a)+\omega(b)$ for all $a,b\in \widehat{\mathbb{Q}}$. Moreover, by a direct computation, we know that $\omega(ab)=\omega(a)\omega(b)$ for all $a,b\in \widehat{\mathbb{Q}}$. We obtain the desired conclusion.
\end{proof}

\begin{Prop}
Let $a\in \widehat{\mathbb{Q}}$. Then $a$ is invertible if and only if $\omega(a)$ is invertible, and $(\omega(a))^{-1}=\omega(a^{-1})\in \mathbb{Z}$.
\end{Prop}

\begin{proof}
Let $a=a_{\sf st}+a_{\sf in}\epsilon$. It is obvious that $\omega(a)=\omega(a_{\sf st})+\omega(a_{\sf in})\epsilon$, and $\omega(a)$ is invertible if and only if $\omega(a_{\sf st})$ is invertible,  which is equivalent to ${\rm det}(\omega(a_{\sf st}))\neq 0$. On the other hand, $a$ is invertible if and only if $a_{\sf st}\neq 0$, which is equivalent to $|a_{\sf st}|^2\neq 0$. Since ${\rm det}(\omega(a_{\sf st}))=|a_{\sf st}|^2$, we know that $a$ is invertible if and only if $\omega(a)$ is invertible. If $a$ is invertible, then from $aa^{-1}=1$, we know $\omega(a)\omega(a^{-1})=I_2$, which means that $(\omega(a))^{-1}=\omega(a^{-1})\in \mathbb{Z}$.
\end{proof}

\begin{Prop}\label{FFFIn}
For any $a\in \widehat{\mathbb{Q}}$, it holds that $1/2\|\omega(a)\|_F\leq|a|\leq \sqrt{2}\|\omega(a)\|_F$.
\end{Prop}

\begin{proof}
Let $a=a_{\sf st}+a_{\sf in}\epsilon$, it is easy to see that $\omega(a)=\omega(a_{\sf st})+\omega(a_{\sf in})\epsilon$. If $a$ is  appreciable, i.e., $a_{\sf st}\neq 0$, then $\omega(a_{\sf st})$ is invertible. By (\ref{FNorm-DQM}) and a direct calculation, we have
$$
\|\omega(a)\|_F=\sqrt{2}|a_{\sf st}|+\frac{2(a_{\sf st}\bar a_{\sf in}+a_{\sf in} \bar a_{\sf st})}{\sqrt{2}|a_{\sf st}|}\epsilon.
$$
On the other hand, by (\ref{e7}), we have
$$
|a|=|a_{\sf st}|+\frac{(a_{\sf st}\bar a_{\sf in}+a_{\sf in} \bar a_{\sf st})}{2|a_{\sf st}|}\epsilon.
$$
From these, we have $1/2\|\omega(a)\|_F\leq |a|\leq \|\omega(a)\|_F$.

If $a$ is infinitesimal, i.e., $a_{\sf st}=0$, then $a=a_{\sf in}\epsilon$, and $|a|=|a_{\sf in}|\epsilon$ by (\ref{e7}). On the other hand, since $\omega(a)=\omega(a_{\sf in})\epsilon$, it holds that $\|\omega(a)\|_F=\|\omega(a_{\sf in})\|_F\epsilon$ by (\ref{FNorm-DQM}). It is easy to see that $\|\omega(a_{\sf in})\|_F=\sqrt{2}|a_{\sf in}|$. Consequently, it holds that $1/2\|\omega(a)\|_F\leq |a|\leq \|\omega(a)\|_F$.
We complete the proof.
\end{proof}

Based upon the map $\omega: \widehat{\mathbb{Q}}\rightarrow \widehat{\mathbb{C}}^{2\times 2}$, define the map $\tilde{\omega}:\widehat{\mathbb{Q}}^{m\times n}\rightarrow\widehat{\mathbb{C}}^{2m\times 2n}$ as follows
$$
\tilde{\omega}(A)=(\omega(a_{ij}))_{i,j=1}^{m,n},
$$
where $a_{ij}\in \widehat{\mathbb{Q}}$ is the $(i,j)$th element of $A$ for $i=1,2,\ldots,m,j=1,2,\ldots,n$. From this  definition, we know $\tilde{\omega}(A)=\tilde{\omega}(A_{\sf st})+\tilde{\omega}(A_{\sf in})\epsilon$ when $A=A_{\sf st}+A_{\sf in}\epsilon$ with $A_{\sf st}, A_{\sf in}\in \mathbb{Q}^{m\times n}$.

\begin{Prop}\label{Omg-prerty} We have

(a) $\tilde{\omega}(\alpha A+\beta B)=\alpha\tilde{\omega}(A)+\beta\tilde{\omega}(B)$ for all $A,B\in \widehat{\mathbb{Q}}^{m\times n}$ and $\alpha,\beta\in \widehat{\mathbb{R}}$.

(b) $\tilde{\omega}(A^*)=(\tilde{\omega}(A))^*$ for all $A\in \widehat{\mathbb{Q}}^{m\times n}$.

(c) $\tilde{\omega}(AB)=\tilde{\omega}(A)\tilde{\omega}(B)$ for all $A\in \widehat{\mathbb{Q}}^{m\times p}$ and $B\in \widehat{\mathbb{Q}}^{p\times n}$.

(d) If $A\in \widehat{\mathbb{Q}}^{m\times m}$ is invertible, then $\tilde{\omega}(A)$ is invertible, and $(\tilde{\omega}(A))^{-1}=\tilde{\omega}(A^{-1})$.
\end{Prop}

\begin{proof}
The first two conclusions can be proved by the definition of $\tilde{\omega}$. Now we prove the conclusion (c). Let $A=(a_{ik})$, $B=(b_{kj})$. Denote $C=(c_{ij})=AB$, which means that $c_{ij}=\sum_{k=1}^pa_{ik}b_{kj}$ for any $i=1,2,\ldots, m$ and $j=1,2,\ldots,n$. Consequently, we know that $\omega(c_{ij})=\sum_{k=1}^p\omega(a_{ik})\omega(b_{kj})$, which implies, together with the multiplication rules for blocked matrices, that the conclusion (c) holds. If $A\in \widehat{\mathbb{Q}}^{m\times m}$ is invertible, i.e, there exists $A^{-1}\in \widehat{\mathbb{Q}}^{m\times m}$ such that $AA^{-1}=I_m$, then by (c), we have $\tilde{\omega}(A)\tilde{\omega}(A^{-1})=\tilde{\omega}(I_m)=I_{2m}$, which means that the conclusion (d) holds. We complete the proof. 
\end{proof}

From Proposition \ref{Omg-prerty},  Theorem 4.1 in \cite{QL21} and the definition of $\tilde{\omega}(\cdot)$, we immediately have
\begin{Prop}
A matrix $A\in \widehat{\mathbb{Q}}^{m\times m}$ is hermitian, skewhermitian, unitary, normal, positive definite, or positive
semidefinite,  if and only if $\tilde{\omega}(A)$ is such.
\end{Prop}

\begin{Prop} We have

(a) There exist positive appreciable constants $\tau,\kappa\in \widehat{\mathbb{R}}_{++}$ such that
$$
\tau\|\tilde{\omega}(A)\|_F\leq\|A\|_F\leq \kappa\|\tilde{\omega}(A)\|_F
$$
for every $A\in \widehat{\mathbb{Q}}^{m\times n}$.

(b) The map $\tilde{\omega}$ is a unital isomorphism of the set $\widehat{\mathbb{Q}}^{m\times m}$ onto the subset
    $\mathbb{Z}^{m\times m}$ of $\widehat{\mathbb{C}}^{2m\times 2m}$, where $\mathbb{Z}^{m\times m}:=\left\{(Z_{i,j})_{i,j=1}^m~|~Z_{i,j}\in \mathbb{Z}\right\}$.
\end{Prop}

\begin{proof}
It is obvious that $A$ is appreciable, if and only if $\tilde{\omega}(A)$ is appreciable. If $A$ is appreciable, then $$\|A\|_F^2=\sum_{i=1}^m \sum_{j=1}^n |a_{ij}|^2~~~~{\rm and}~~~~\|\tilde{\omega}(A)\|_F^2=\sum_{i=1}^m \sum_{j=1}^n \|\omega(a_{ij})\|_F^2.$$
Consequently, it holds that $\|\omega(a_{ij})\|_F^2=2|a_{ij}|^2$ for every $i,j$, which implies $$1/2\|\tilde{\omega}(A)\|_F\leq\|A\|_F\leq \sqrt{2}\|\tilde{\omega}(A)\|_F.$$

If $A$ is infinitesimal, then $A=A_{\sf in}\epsilon$ and $\|A\|_F=\|A_{\sf in}\|_F\epsilon$. In this case, since $\tilde{\omega}(A)=\tilde{\omega}(A_{\sf in})\epsilon$, we know $\|\tilde{\omega}(A)\|_F=\|\tilde{\omega}(A_{\sf in})\|_F\epsilon$. Since $A_{\sf in}\in \mathbb{Q}^{m\times n}$, by Property \ref{FFFIn}, $1/2\|\tilde{\omega}(A_{\sf in})\|_F\leq\|A_{\sf in}\|_F\leq \sqrt{2}\|\tilde{\omega}(A_{\sf in})\|_F$. We obtain the desired first conclusion.

The second conclusion follows from Proposition \ref{UnitIsomMap}. We complete the proof.
\end{proof}

\begin{Prop}\label{AX-B}
Let $A\in \mathbb{Z}^{m\times m}$ be invertible. Then for any $B=(B_1;B_2;\ldots;B_m)$ with $B_1, B_2,\ldots,B_m\in \mathbb{Z}$, the equation $AX=B$ has unite solution in $\mathbb{Z}^{m\times 1}$.
\end{Prop}

\begin{proof}
Write $A=(A_{ij})$ with $A_{ij}\in \mathbb{Z}$ for $i,j=1,2,\ldots,m$. Since ${\rm det}[A_{\rm st}]\neq 0$ where $A_{\rm st}=((A_{ij})_{\rm st})$, we claim that there must be at last one $A_{i1}=(A_{i1})_{\rm st}+(A_{i1})_{\rm in}\epsilon$ in $\{A_{11},A_{21},\ldots,A_{m1}\}$, such that ${\rm det}[(A_{i1})_{\rm st})]\neq 0$. Otherwise, we obtain $(A_{i1})_{\rm st}=O$ for every $i=1,2,\ldots,m$, from the special structure of $(A_{i1})_{\rm st}$. Consequently, it holds that ${\rm det}[A_{\rm st}]=0$ from $A_{i1}=(A_{i1})_{\rm in}\epsilon$ for $i=1,2,\ldots,m$, which is a contradiction. Without loss of generality, we assume ${\rm det}[(A_{11})_{\rm st}]\neq 0$, which implies $A_{11}^{-1}$ exists. 
Consequently, similar to common block matrices, we apply the block row transformations to the augmented matrix $(A,B)$ of the equation $AX=B$, such that all $A_{21},\ldots, A_{m1}$ in the first block column were transformed into $O$, i.e.,
$$
(A,B)\rightarrow\left(
\begin{array}{cccccc}
A_{11}&A_{12}&A_{13}&\cdots&A_{1m}&B_1\\
O&A^\prime_{22}&A^\prime_{23}&\cdots&A^\prime_{2m}&B^\prime_2\\
O&A^\prime_{32}&A^\prime_{33}&\cdots&A^\prime_{3m}&B^\prime_2\\
\vdots&\vdots&\vdots&\ddots&\vdots&\vdots\\
O&A^\prime_{m2}&A^\prime_{m3}&\cdots&A^\prime_{mm}&B^\prime_m\\
\end{array}
\right),
$$
where $A^\prime_{ij}=A_{ij}-A_{i1}A_{11}^{-1}A_{1j}$ and $B^\prime_{i}=B_i-A_{i1}A_{11}^{-1}B_{1}$ for $i,j=2,\ldots,m$. Due to the special structure of $A_{ij}$ and $B_i$, it is easy to verify that $A^\prime_{ij}\in \mathbb{Z}$ and $B^\prime_{i}\in \mathbb{Z}$ for $i,j=2,\ldots,m$. Furthermore, by Proposition \ref{Prop-2}, we have $${\rm det}[A_{\rm st}]={\rm det}[(A_{11})_{\rm st}]{\rm det}[\left(
\begin{array}{cccc}
(A^\prime_{22})_{\rm st}&(A^\prime_{23})_{\rm st}&\cdots&(A^\prime_{2m})_{\rm st}\\
(A^\prime_{32})_{\rm st}&(A^\prime_{33})_{\rm st}&\cdots&(A^\prime_{3m})_{\rm st}\\
\vdots&\vdots&\ddots&\vdots\\
(A^\prime_{m2})_{\rm st}&(A^\prime_{m3})_{\rm st}&\cdots&(A^\prime_{mm})_{\rm st}\\
\end{array}
\right)]\neq 0,$$ which implies ${\rm det}[\left(
\begin{array}{cccc}
(A^\prime_{22})_{\rm st}&(A^\prime_{23})_{\rm st}&\cdots&(A^\prime_{2m})_{\rm st}\\
(A^\prime_{32})_{\rm st}&(A^\prime_{33})_{\rm st}&\cdots&(A^\prime_{3m})_{\rm st}\\
\vdots&\vdots&\ddots&\vdots\\
(A^\prime_{m2})_{\rm st}&(A^\prime_{m3})_{\rm st}&\cdots&(A^\prime_{mm})_{\rm st}\\
\end{array}
\right)]\neq 0$. From this and the special structure of $(A^\prime_{i2})_{\rm st}$, where $i=2,3,\ldots,m$, we further claim that there must be at last one $A^\prime_{i2}$ in $\{A^\prime_{22},A^\prime_{32},\ldots,A^\prime_{m2}\}$, such that ${\rm det}[(A^\prime_{i2})_{\rm st})]\neq 0$. Without loss of generality, we assume ${\rm det}[(A^\prime_{22})_{\rm st})]\neq 0$, which implies $A^\prime_{22}$ is invertible. Consequently, similar to the above process, we have
$$
\left(
\begin{array}{cccccc}
A_{11}&A_{12}&A_{13}&\cdots&A_{1m}&B_1\\
O&A^\prime_{22}&A^\prime_{23}&\cdots&A^\prime_{2m}&B^\prime_2\\
O&A^\prime_{32}&A^\prime_{33}&\cdots&A^\prime_{3m}&B^\prime_2\\
\vdots&\vdots&\vdots&\ddots&\vdots&\vdots\\
O&A^\prime_{m2}&A^\prime_{m3}&\cdots&A^\prime_{mm}&B^\prime_m\\
\end{array}
\right)\rightarrow\left(
\begin{array}{cccccc}
A_{11}&O&A^{\prime\prime}_{13}&\cdots&A^{\prime\prime}_{1m}&B^{\prime\prime}_1\\
O&A^{\prime}_{22}&A^{\prime\prime}_{23}&\cdots&A^{\prime\prime}_{2m}&B^{\prime\prime}_2\\
O&O&A^{\prime\prime}_{33}&\cdots&A^{\prime\prime}_{3m}&B^{\prime\prime}_3\\
\vdots&\vdots&\vdots&\ddots&\vdots&\vdots\\
O&O&A^{\prime\prime}_{m3}&\cdots&A^{\prime\prime}_{mm}&B^{\prime\prime}_m\\
\end{array}
\right),
$$
where $A^{\prime\prime}_{1j}=A_{1j}-A_{12}(A^\prime_{22})^{-1}A^\prime_{2j}\in \mathbb{Z}$, $A^{\prime\prime}_{ij}=A^\prime_{ij}-A^\prime_{i2}(A^\prime_{22})^{-1}A^\prime_{2j}\in \mathbb{Z}$, $B^{\prime\prime}_{1}=B_1-A_{12}(A^\prime_{22})^{-1}B^\prime_{2}\in \mathbb{Z}$, and $B^{\prime\prime}_{i}=B^\prime_i-A_{i2}A_{22}^{-1}B^\prime_{2}\in \mathbb{Z}$ for $i,j=3,\ldots,m$. Since ${\rm det}[\left(
\begin{array}{ccc}
(A^{\prime\prime}_{33})_{\rm st}&\cdots&(A^{\prime\prime}_{3m})_{\rm st}\\
\vdots&\ddots&\vdots\\
(A^{\prime\prime}_{m3})_{\rm st}&\cdots&(A^{\prime\prime}_{mm})_{\rm st}\\
\end{array}
\right)]\neq 0$, by continuously implementing block
row transformations, we finally have
$$
\left(
\begin{array}{cccccc}
A_{11}&O&A^{\prime\prime}_{13}&\cdots&A^{\prime\prime}_{1m}&B^{\prime\prime}_1\\
O&A^{\prime}_{22}&A^{\prime\prime}_{23}&\cdots&A^{\prime\prime}_{2m}&B^{\prime\prime}_2\\
O&O&A^{\prime\prime}_{33}&\cdots&A^{\prime\prime}_{3m}&B^{\prime\prime}_3\\
\vdots&\vdots&\vdots&\ddots&\vdots&\vdots\\
O&O&A^{\prime\prime}_{m3}&\cdots&A^{\prime\prime}_{mm}&B^{\prime\prime}_m\\
\end{array}
\right)\rightarrow\cdots\rightarrow\left(
\begin{array}{cccccc}
A_{11}&O&O&\cdots&O&B^{\prime\prime}_1\\
O&A^{\prime}_{22}&O&\cdots&O&B^{\prime\prime}_2\\
O&O&A^{\prime\prime}_{33}&\cdots&O&B^{\prime\prime}_3\\
\vdots&\vdots&\vdots&\ddots&\vdots&\vdots\\
O&O&O&\cdots&A^{\prime\prime}_{mm}&B^{\prime\prime}_m\\
\end{array}
\right),
$$
where $A_{11}$, $A^{\prime}_{22}$ and all $A^{\prime\prime}_{ii}~(i=3,\ldots,m)$ are invertible. It is easy to see that the resulting augmented matrix corresponds to the following matrix equation
 $$
 \left\{
  \begin{array}{l}
  A_{11}X_1=B^{\prime\prime}_1\\
  A^{\prime}_{22}X_2=B^{\prime\prime}_2\\
  A^{\prime\prime}_{ii}X_i=B^{\prime\prime}_i,~i=3,\ldots,m,
  \end{array}
  \right.
 $$
 which has the same solution to $AX=B$. Hence, we obtain
$X_1=A_{11}^{-1}B^{\prime\prime}_1$, $X_2=(A^\prime_{22})^{-1}B^{\prime\prime}_2$ and $X_i=(A^{\prime\prime}_{ii})^{-1}B^{\prime\prime}_i$ for $i=3,\ldots,m$. Finally, due to the fact that $A_{11}^{-1}, B^{\prime\prime}_1, (A^{\prime\prime}_{ii})^{-1}, B^{\prime\prime}_i\in \mathbb{Z}$ for $i=2,\ldots,m$, we know $X_i\in \mathbb{Z}$ for $i=1,2,\ldots,m$.
\end{proof}

Based on dual complex matrix representations of dual quaternion matrices, we define the quasi-determinant of $A\in \widehat{\mathbb{Q}}^{m\times m}$ as follows:
$$
{\rm det}_q (A) := {\rm det}(\tilde{\omega}(A)),
$$
where in the right-hand side we use the standard determinant definition for the determinant of $2m\times 2m$ dual complex matrices.

From this definition, we know ${\rm det}_q (\alpha) = |\alpha|^2$  for all $\alpha\in \widehat{\mathbb{Q}}$, which means that ${\rm det}_q (\alpha) =0$ for any infinitesimal dual quaternion $\alpha$. More general, if some row or column of $A$ is infinitesimal, then we have ${\rm det}_q (A)=0$.

\begin{Prop}\label{DQProp-2}
Let $A\in \widehat{\mathbb{Q}}^{m\times m}$, $B\in \widehat{\mathbb{Q}}^{n\times n}$, $C\in \widehat{\mathbb{Q}}^{n\times m}$ and $D\in \widehat{\mathbb{Q}}^{m\times n}$. It holds that
$${\rm det}_q\left(\begin{array}{cc}
A&O\\
C&B
\end{array}
\right)={\rm det}_q[A]{\rm det}_q[B]~~~{\rm and}~~~{\rm det}_q\left(\begin{array}{cc}
A&D\\
O&B
\end{array}
\right)={\rm det}_q[A]{\rm det}_q[B].$$
\end{Prop}

\begin{proof}
It follows from Proposition \ref{Prop-2} and the definition of ${\rm det}_q(\cdot)$.
\end{proof}

\begin{Prop}\label{DQAB-Product}
Let $C=AB$, where $A,B\in \widehat{\mathbb{Q}}^{m\times m}$. Then it holds that ${\rm det}_q[C]={\rm det}_q[A]{\rm det}_q[B]$.
\end{Prop}
\begin{proof}
It follows from By Proposition \ref{Omg-prerty}, Proposition \ref{AB-Product},  and the definition of ${\rm det}_q(\cdot)$.
\end{proof}

\begin{Thm}
Let $A=A_{\rm st}+A_{\rm in
}\epsilon\in \widehat{\mathbb{Q}}^{m\times m}$. Then, $A$ is invertible, if and only if ${\rm det}_q(A)$ is appreciable, i.e., ${\rm det}_q(A_{\rm st})\neq 0$.
\end{Thm}
\begin{proof}
If $A$ is invertible, the there exists a $B\in \widehat{\mathbb{Q}}^{m\times m}$ such that $AB=I$, which implies, together with Proposition \ref{Omg-prerty}, that $\tilde{\omega}(A)\tilde{\omega}(B)=I_{2m}$. Moreover, by Proposition \ref{AB-Product}, we know that ${\rm det}_q(A)= {\rm det}(\tilde{\omega}(A))$ is appreciable.

Conversely, if ${\rm det}_q(A)$ is appreciable, by (\ref{Inver-mat}) and (\ref{Det-A}), we know that $\tilde{\omega}(A)$ is invertible. Consequently, since $\tilde{\omega}(A)\in \mathbb{Z}^{m\times m}$, by Proposition \ref{AX-B}, there exists $C\in \mathbb{Z}^{m\times m}$ such that $\tilde{\omega}(A)C=I_{2m}$, which implies, together with the fact that $\omega(\cdot)$ is one-to-one and onto map on the set $\mathbb{Z}$,  that there exists a $B\in \widehat{\mathbb{Q}}^{m\times m}$ such that $C=\tilde{\omega}(B)$. Finally, from $\tilde{\omega}(A)\tilde{\omega}(B)=I_{2m}$, we know $AB=I_m$, which means that $A$ is invertible. We complete the proof.
\end{proof}

Now we introduce the conception of the quasi-characteristic polynomial of dual quaternion matrices. For given $A=A_{\sf st}+A_{\sf in}\epsilon\in \widehat{\mathbb{Q}}^{m\times m}$, define $f^q_A:\widehat{\mathbb{Q}}\rightarrow\widehat{\mathbb{C}}$ by
$$
f_A^q(\lambda)={\rm det}_q[\lambda I-A],
$$
where $\lambda=\lambda_{\sf st}+\lambda_{\sf in}\epsilon\in \widehat{\mathbb{Q}}$. We call the $f^q_A$ above the quasi-characteristic polynomial of $A$, and call $\lambda=\lambda_{\sf st}+\lambda_{\sf in}\epsilon\in \widehat{\mathbb{Q}}$ satisfying $f^q_A(\lambda)=0$ the quasi-characteristic root of $A$.
\begin{Prop}\label{REig-Root}
Let $A=A_{\sf st}+A_{\sf in}\epsilon\in \widehat{\mathbb{Q}}^{m\times m}$. The right eigenvalues of $A$ must be the quasi-characteristic roots of $A$.
\end{Prop}
\begin{proof}
Take any right eigenvalue $\bar{\lambda}$ of $A$,  with ${\bf u}$ as an associated right eigenvector. Denote ${\bf v}_1={\bf u}/\|{\bf u}\|$. It is clear that ${\bf v}_1\in \widehat{\mathbb{Q}}^{m\times 1}$. By Corollary 3.10 in \cite{LHQ22}, there exists $U_2\in \widehat{\mathbb{Q}}^{m\times (m-1)}$ such that $U=({\bf v}_1, U_2)$ is unitary. By Proposition \ref{AB-Product}, It is obvious that $U^*AU=B$, where
$$
B=\left(
\begin{array}{cc}
\bar{\lambda} &B_{12}\\
O&B_{22}
\end{array}\right)~{\rm with}~~B_{12}\in \widehat{\mathbb{Q}}^{1\times (m-1)}~~{\rm and}~~ B_{22}\in \widehat{\mathbb{Q}}^{(m-1)\times (m-1)}.
$$
Since $U^*AU=B$, we know $U^*(\lambda I-A)U=\lambda I-B$, which implies, together with (c) in Proposition \ref{Omg-prerty}, that $\tilde{\omega}(\lambda I-B)=\tilde{\omega}(U^*)\tilde{\omega}(\lambda I-A)\tilde{\omega}(U)$. By Proposition \ref{AB-Product}, it holds that
$$f_B^q(\lambda)={\rm det}_q[\lambda I-B]={\rm det}[\tilde{\omega}(U^*)]{\rm det}[\tilde{\omega}(U)]{\rm det}[\tilde{\omega}(\lambda I-A)].$$
Since ${\rm det}[\tilde{\omega}(U^*)]{\rm det}[\tilde{\omega}(U)]={\rm det}[\tilde{\omega}(U^*)\tilde{\omega}(U)]={\rm det}[\tilde{\omega}(U^*U)]={\rm det}[I_{2m}]=1$, we have $f_B^q(\lambda)=f_A^q(\lambda)$. On the other hand, since $$\tilde{\omega}(\lambda I-B)=\left(
\begin{array}{cc}
\omega(\lambda-\bar{\lambda}) &-\tilde{\omega}(B_{12})\\
O&\tilde{\omega}(\lambda I_{m-1}-B_{22})
\end{array}\right),$$ by Proposition \ref{Prop-2}, we know
$f_B^q(\lambda)={\rm det}[\omega(\lambda-\bar{\lambda})]{\rm det}_q[\lambda I_{m-1}-B_{22}]$. Hence, we have $f_A^q(\lambda)={\rm det}[\omega(\lambda-\bar{\lambda})]{\rm det}_q[\lambda I_{m-1}-B_{22}]=|\lambda-\bar{\lambda}|^2{\rm det}_q[\lambda I_{m-1}-B_{22}]$. We obtain the desired conclusion and complete the proof.
\end{proof}

\begin{Thm}\label{q-Lambda-f} Let $A, B\in \widehat{\mathbb{Q}}^{m\times m}$. If $A\sim B$, i.e., $A = P
^{-1}BP$ for some invertible matrix $P\in \widehat{\mathbb{Q}}^{m\times m}$, then $f_A^q(\lambda)=f_B^q(\lambda)$.
\end{Thm}

\begin{proof}
Since $A = P^{-1}BP$, it holds that $\lambda I-A= P^{-1}(\lambda I-B)P$. By Proposition \ref{DQAB-Product}, we have
$$
f_A^q(\lambda)={\rm det}_q[\lambda I-A]={\rm det}_q(P^{-1}){\rm det}_q(\lambda I-B){\rm det}_q(P)={\rm det}_q(P^{-1}){\rm det}_q(P)f_{B,q}(\lambda)=f_{B,q}(\lambda),
$$
where the last equality is due to the fact ${\rm det}_q(P^{-1}){\rm det}_q(P)=1$, which comes from $P^{-1}P=I$ and Proposition \ref{DQAB-Product}.
\end{proof}

\begin{Prop}
Let $A\in \widehat{\mathbb{Q}}^{m\times m}$ be Hermitian. It holds that ${\rm det}_q[A]=\prod_{i=1}^m|\lambda_i|^2$, where $\lambda_1\geq\lambda_2\geq\ldots\geq\lambda_m$ are the eigenvalues of $A$.
\end{Prop}
\begin{proof}
By Theorem 4.1 in \cite{QL21}, there are unitary matrix $U\in \widehat{\mathbb{Q}}^{m\times m}$ and a diagonal matrix $\Sigma \in \widehat{\mathbb{R}}^{m\times m}$ such that $A=U\Sigma U^*$, where $\Sigma={\rm diag}(\lambda_1,\lambda_2,\ldots,\lambda_m)$. Consequently, by Proposition \ref{DQAB-Product}, it holds that ${\rm det}_q[A]={\rm det}_q[U]{\rm det}_q[\Sigma]{\rm det}_q[ U^*]={\rm det}_q[\Sigma]$. Furthermore, by Proposition \ref{DQProp-2}, we have ${\rm det}_q[\Sigma]=\prod_{i=1}^m{\rm det}_q(\lambda_i)=\prod_{i=1}^m|\lambda_i|^2$, since ${\rm det}_q(\alpha)=|\alpha|^2$ for any $\alpha\in \widehat{\mathbb{R}}$. Hence, we obtain the desired result and complete the proof.
\end{proof}

\section{Final Remarks}\label{Conclusion}
In this paper, we studied determinants of dual complex matrices, quasi-determinants of dual quaternion matrices, and the relationship between the right eigenvalues and the quasi-characteristic roots of dual quaternion matrices, as well as some basic properties of determinants and quasi-determinants mentioned above. First, we studied some basic properties of determinants of dual complex matrices, including Sturm theorem and Bloomfield-Watson inequality for dual complex matrices. Then, we showed that every eigenvalue of a dual complex matrix must be the root of the characteristic polynomial of this matrix, however, its reverse proposition does not hold. Furthermore, 
we introduced the concept of quasi-determinants of dual quaternion matrices, and showed that every right eigenvalue of a dual quaternion matrix must be the root of the quasi-characteristic polynomial of this matrix, as well as the quasi-determinant of a dual quaternion Hermitian matrix is equivalent to the product of the square of the magnitudes of all eigenvalues. As a new area of applied mathematics, our results enrich the basic theory of dual quaternion matrices. We also hope that our results may play some important roles in the applications of rigid body motion and multi-agent formation control.

\end{document}